\documentclass[11pt,twoside,a4paper]{article}
\usepackage{amsmath, amssymb}
\usepackage[latin1]{inputenc}
\usepackage[T1]{fontenc}   
\usepackage{stmaryrd}
\usepackage[english]{babel}
\newcommand{\tun}{\begin{picture}(5,0)(-2,-1)
\put(0,0){\circle*{2}}
\end{picture}}

\newcommand{\tdeux}{\begin{picture}(7,7)(0,-1)
\put(3,0){\circle*{2}}
\put(3,5){\circle*{2}}
\put(3,0){\line(0,1){5}}
\end{picture}}

\newcommand{\ttroisun}{\begin{picture}(15,12)(-5,-1)
\put(3,0){\circle*{2}}
\put(6,7){\circle*{2}}
\put(0,7){\circle*{2}}
\put(-0.65,0){$\vee$}
\end{picture}}

\newcommand{\ttroisdeux}{\begin{picture}(5,15)(-2,-1)
\put(0,0){\circle*{2}}
\put(0,5){\circle*{2}}
\put(0,10){\circle*{2}}
\put(0,0){\line(0,1){5}}
\put(0,5){\line(0,1){5}}
\end{picture}}

\newcommand{\tquatreun}{\begin{picture}(15,12)(-5,-1)
\put(3,0){\circle*{2}}
\put(6,7){\circle*{2}}
\put(0,7){\circle*{2}}
\put(3,7){\circle*{2}}
\put(-0.65,0){$\vee$}
\put(3,0){\line(0,1){7}}
\end{picture}}

\newcommand{\tquatredeux}{\begin{picture}(15,18)(-5,-1)
\put(3,0){\circle*{2}}
\put(6,7){\circle*{2}}
\put(0,7){\circle*{2}}
\put(0,14){\circle*{2}}
\put(-0.65,0){$\vee$}
\put(0,7){\line(0,1){7}}
\end{picture}}

\newcommand{\tquatrequatre}{\begin{picture}(15,18)(-5,-1)
\put(3,5){\circle*{2}}
\put(6,12){\circle*{2}}
\put(0,12){\circle*{2}}
\put(3,0){\circle*{2}}
\put(-0.65,5){$\vee$}
\put(3,0){\line(0,1){5}}
\end{picture}}

\newcommand{\tquatrecinq}{\begin{picture}(9,19)(-2,-1)
\put(0,0){\circle*{2}}
\put(0,5){\circle*{2}}
\put(0,10){\circle*{2}}
\put(0,15){\circle*{2}}
\put(0,0){\line(0,1){5}}
\put(0,5){\line(0,1){5}}
\put(0,10){\line(0,1){5}}
\end{picture}}

\newcommand{\tcinqun}{\begin{picture}(20,8)(-5,-1)
\put(3,0){\circle*{2}}
\put(-7,5){\circle*{2}}
\put(13,5){\circle*{2}}
\put(6,7){\circle*{2}}
\put(0,7){\circle*{2}}
\put(-0.5,0){$\vee$}
\put(3,0){\line(2,1){10}}
\put(3,0){\line(-2,1){10}}
\end{picture}}

\newcommand{\tcinqdeux}{\begin{picture}(15,14)(-5,-1)
\put(3,0){\circle*{2}}
\put(0,14){\circle*{2}}
\put(6,7){\circle*{2}}
\put(0,7){\circle*{2}}
\put(3,7){\circle*{2}}
\put(-0.65,0){$\vee$}
\put(3,0){\line(0,1){7}}
\put(0,7){\line(0,1){7}}
\end{picture}}

\newcommand{\tcinqcinq}{\begin{picture}(15,19)(-5,-1)
\put(3,0){\circle*{2}}
\put(6,7){\circle*{2}}
\put(0,7){\circle*{2}}
\put(6,14){\circle*{2}}
\put(0,14){\circle*{2}}
\put(-0.65,0){$\vee$}
\put(6,7){\line(0,1){7}}
\put(0,7){\line(0,1){7}}
\end{picture}}

\newcommand{\tcinqsept}{\begin{picture}(15,8)(-5,-1)
\put(3,0){\circle*{2}}
\put(6,7){\circle*{2}}
\put(0,7){\circle*{2}}
\put(3,14){\circle*{2}}
\put(9,14){\circle*{2}}
\put(-0.65,0){$\vee$}
\put(2.35,7){$\vee$}
\end{picture}}

\newcommand{\tcinqneuf}{\begin{picture}(15,26)(-5,-1)
\put(3,0){\circle*{2}}
\put(6,7){\circle*{2}}
\put(0,7){\circle*{2}}
\put(6,14){\circle*{2}}
\put(6,21){\circle*{2}}
\put(-0.65,0){$\vee$}
\put(6,7){\line(0,1){7}}
\put(6,14){\line(0,1){7}}
\end{picture}}

\newcommand{\tcinqdix}{\begin{picture}(15,19)(-5,-1)
\put(3,5){\circle*{2}}
\put(6,12){\circle*{2}}
\put(0,12){\circle*{2}}
\put(3,0){\circle*{2}}
\put(3,12){\circle*{2}}
\put(-0.5,5){$\vee$}
\put(3,0){\line(0,1){12}}
\end{picture}}

\newcommand{\tcinqonze}{\begin{picture}(15,26)(-5,-1)
\put(3,5){\circle*{2}}
\put(6,12){\circle*{2}}
\put(0,12){\circle*{2}}
\put(3,0){\circle*{2}}
\put(0,19){\circle*{2}}
\put(-0.65,5){$\vee$}
\put(3,0){\line(0,1){5}}
\put(0,12){\line(0,1){7}}
\end{picture}}

\newcommand{\tcinqtreize}{\begin{picture}(5,26)(-2,-1)
\put(0,0){\circle*{2}}
\put(0,7){\circle*{2}}
\put(0,0){\line(0,1){7}}
\put(0,14){\circle*{2}}
\put(-3,21){\circle*{2}}
\put(3,21){\circle*{2}}
\put(-3.65,14){$\vee$}
\put(0,7){\line(0,1){7}}
\end{picture}}

\newcommand{\tcinqquatorze}{\begin{picture}(9,26)(-5,-1)
\put(0,0){\circle*{2}}
\put(0,5){\circle*{2}}
\put(0,10){\circle*{2}}
\put(0,15){\circle*{2}}
\put(0,20){\circle*{2}}
\put(0,0){\line(0,1){5}}
\put(0,5){\line(0,1){5}}
\put(0,10){\line(0,1){5}}
\put(0,15){\line(0,1){5}}
\end{picture}}

\newcommand{\tdun}[1]{\begin{picture}(10,5)(-2,-1)
\put(0,0){\circle*{2}}
\put(3,-2){\tiny #1}
\end{picture}}

\newcommand{\tddeux}[2]{\begin{picture}(12,5)(0,-1)
\put(3,0){\circle*{2}}
\put(3,5){\circle*{2}}
\put(3,0){\line(0,1){5}}
\put(6,-2){\tiny #1}
\put(6,3){\tiny #2}
\end{picture}}

\newcommand{\tdtroisun}[3]{\begin{picture}(20,12)(-5,-1)
\put(3,0){\circle*{2}}
\put(6,7){\circle*{2}}
\put(0,7){\circle*{2}}
\put(-0.65,0){$\vee$}
\put(5,-2){\tiny #1}
\put(9,5){\tiny #2}
\put(-5,5){\tiny #3}
\end{picture}}

\newcommand{\tdtroisdeux}[3]{\begin{picture}(12,15)(-2,-1)
\put(0,0){\circle*{2}}
\put(0,5){\circle*{2}}
\put(0,10){\circle*{2}}
\put(0,0){\line(0,1){5}}
\put(0,5){\line(0,1){5}}
\put(3,-2){\tiny #1}
\put(3,3){\tiny #2}
\put(3,9){\tiny #3}
\end{picture}}

\newcommand{\tdquatreun}[4]{\begin{picture}(20,12)(-5,-1)
\put(3,0){\circle*{2}}
\put(6,7){\circle*{2}}
\put(0,7){\circle*{2}}
\put(3,7){\circle*{2}}
\put(-0.6,0){$\vee$}
\put(3,0){\line(0,1){7}}
\put(5,-2){\tiny #1}
\put(8.5,5){\tiny #2}
\put(1,10){\tiny #3}
\put(-5,5){\tiny #4}
\end{picture}}

\newcommand{\ptroisun}{\begin{picture}(15,12)(-5,-1)
\put(3,7){\circle*{2}}
\put(-0.65,0){$\wedge$}
\put(6,0){\circle*{2}}
\put(0,0){\circle*{2}}
\end{picture}}

\newcommand{\pquatreun}{\begin{picture}(15,12)(-5,-1)
\put(3,7){\circle*{2}}
\put(-0.65,0){$\wedge$}
\put(6,0){\circle*{2}}
\put(0,0){\circle*{2}}
\put(3,0){\circle*{2}}
\put(2.9,0){\line(0,1){7}}
\end{picture}}
\newcommand{\pquatredeux}{\begin{picture}(15,18)(-5,-1)
\put(3,14){\circle*{2}}
\put(-0.65,7){$\wedge$}
\put(6,7){\circle*{2}}
\put(0,7){\circle*{2}}
\put(0,0){\circle*{2}}
\put(0,0){\line(0,1){7}}
\end{picture}}

\newcommand{\pquatrequatre}{\begin{picture}(15,18)(-5,-1)
\put(3,7){\circle*{2}}
\put(-0.65,0){$\wedge$}
\put(6,0){\circle*{2}}
\put(0,0){\circle*{2}}
\put(3,12){\circle*{2}}
\put(3,7){\line(0,1){5}}
\end{picture}}

\newcommand{\pquatresix}{\begin{picture}(15,9)(-5,-1)
\put(0,0){\circle*{2}}
\put(7,0){\circle*{2}}
\put(0,7){\circle*{2}}
\put(7,7){\circle*{2}}
\put(0,0){\line(0,1){7}}
\put(7,0){\line(0,1){7}}
\put(0,1.5){$\scriptstyle \diagdown$}
\end{picture}}
\newcommand{\pquatresept}{\begin{picture}(15,9)(-5,-1)
\put(0,0){\circle*{2}}
\put(7,0){\circle*{2}}
\put(0,7){\circle*{2}}
\put(7,7){\circle*{2}}
\put(0,0){\line(0,1){7}}
\put(7,0){\line(0,1){7}}
\put(.5,1.5){$\scriptstyle \diagup$}
\put(0,1.5){$\scriptstyle \diagdown$}
\end{picture}}
\newcommand{\pquatrehuit}{\begin{picture}(15,18)(-5,-1)
\put(3,0){\circle*{2}}
\put(-0.65,0){$\vee$}
\put(6,7){\circle*{2}}
\put(0,7){\circle*{2}}
\put(3,14){\circle*{2}}
\put(-0.65,7){$\wedge$}
\end{picture}}

\newcommand{\pdtroisun}[3]{\begin{picture}(23,12)(-7,-1)
\put(3,7){\circle*{2}}
\put(-0.65,0){$\wedge$}
\put(6,0){\circle*{2}}
\put(0,0){\circle*{2}}
\put(5,5){\tiny #1}
\put(-7,-2){\tiny #2}
\put(9,-2){\tiny #3}
\end{picture}}

\input{xy}
\xyoption{all}

\newtheorem{defi}{\indent Definition}
\newtheorem{lemma}[defi]{\indent Lemma}
\newtheorem{cor}[defi]{\indent Corollary}
\newtheorem{theo}[defi]{\indent Theorem}
\newtheorem{prop}[defi]{\indent Proposition}

\newenvironment{proof}{\textbf{Proof.}}{\hfill $\Box$}

\newcommand{\tdelta}{\tilde{\Delta}}
\newcommand{\N}{\mathbb{N}}
\newcommand{\K}{\mathbb{K}}
\newcommand{\QP}{\mathbf{QP}}
\newcommand{\qp}{\mathbf{qp}}
\renewcommand{\P}{\mathbf{P}}
\newcommand{\p}{\mathbf{p}}
\newcommand{\h}{\mathcal{H}}
\newcommand{\isoclass}[1]{\lfloor #1 \rfloor}
\newcommand{\WQSym}{\mathbf{WQSym}}
\newcommand{\bfPW}{\mathbf{PW}}

\begin{document}

\title{Commutative and non-commutative bialgebras of quasi-posets and applications to Ehrhart polynomials}
\date{}
\author{Lo\"\i c Foissy\\ \\
{\small \it Fédération de Recherche Mathématique du Nord Pas de Calais FR 2956}\\
{\small \it Laboratoire de Mathématiques Pures et Appliquées Joseph Liouville}\\
{\small \it Université du Littoral Côte d'Opale-Centre Universitaire de la Mi-Voix}\\ 
{\small \it 50, rue Ferdinand Buisson, CS 80699,  62228 Calais Cedex, France}\\ \\
{\small \it Email: foissy@lmpa.univ-littoral.fr}}

\maketitle

\begin{abstract}
To any poset or quasi-poset is attached a lattice polytope, whose Ehrhart polynomial we study from a Hopf-algebraic point of view.
We use for this two interacting bialgebras on quasi-posets. The Ehrhart polynomial defines a Hopf algebra morphism taking its values
in $\mathbb{Q}[X]$; we deduce from the interacting bialgebras an algebraic proof of the duality principle, a generalization
and a new proof of a result on B-series due to Whright and Zhao, using a monoid of characters on quasi-posets,
and a generalization of Faulhaber's formula.

We also give non-commutative versions of these results: polynomials are replaced by packed words.
We obtain in particular a non-commutative duality principle.
\end{abstract}



\tableofcontents

\section*{Introduction}

Let $P$ be a lattice polytope, that is to say that all its vertices are in $\mathbb{Z}^n$.
The Ehrhart polynomial $ehr^{cl}_P(X)$ is the unique polynomial such that, for all $k \geq 1$, $ehr^{cl}_P(k)$ is the number of points in $\mathbb{Z}^n\cap kP$,
where $kP$ is the image of $P$ by the homothety of center $0$ and ratio $k$.
For example, if $S$ is the square $[0,1]^n$ and $T$ is the triangle of vertices $(0,0)$, $(1,0)$ and $(1,1)$:
\begin{align*}
ehr^{cl}_S(X)&=(X+1)^2,&ehr^{cl}_T(X)&=\frac{(X+1)(X+2)}{2}.
\end{align*}
These polynomials satisfy the reciprocity principle: for all $k \geq 1$, $(-1)^{dim(P)}ehr^{cl}(-k)$ is the number of points of
$\mathbb{Z}^n\cap k \dot{P}$, where $\dot{P}$ is the interior of $P$. For example:
\begin{align*}
ehr^{cl}_S(-X)&=(X-1)^2,&ehr^{cl}_T(-X)&=\frac{(X-1)(X-2)}{2}.
\end{align*}
We refer to \cite{BR} for general results on Ehrhart polynomials.\\

It turns out that these polynomials appear in the theory of B-series (B for Butcher \cite{Butcher}), as explained in \cite{Brouder,Chapoton}.
We now consider rooted trees:
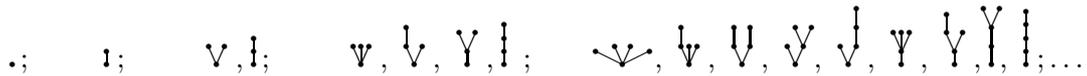
\begin{align*}
&\tun;&&\tdeux;&&\ttroisun,\ttroisdeux;&&\tquatreun,\tquatredeux,\tquatrequatre,\tquatrecinq;&&
\tcinqun,\tcinqdeux,\tcinqcinq,\tcinqsept,\tcinqneuf,\tcinqdix,\tcinqonze,\tcinqtreize,\tcinqquatorze;\ldots
\end{align*}

If $t$ is a rooted tree, we orient its edges from the root to the leaves. If $i,j$ are two vertices of $t$, we shall write
$i\stackrel{t}{\rightarrow} j$ if there is an edge from $i$ to $j$ in $t$.

To any rooted tree $t$, whose vertices are indexed by $1\ldots n$, we associate a lattice polytope $pol(t)$ in a following way:
\begin{align*}
pol(t)&=\left\{(x_1,\ldots,x_n)\in [0,1]^n\mid \forall \:1\leq i,j \leq n, (i\stackrel{t}{\rightarrow} j)\Longrightarrow (x_i\leq x_j)\right\}
\end{align*}
For example, if $t=\tdeux$, indexed as $\tddeux{$1$}{$2$}$, then $pol(t)=T$.

We can consider the Ehrhart polynomial $ehr^{cl}_{pol(t)}(X)$, which we shall simply denote by $ehr^{cl}_t(X)$: for all $k \geq 1$,
\begin{align*}
ehr^{cl}_t(k)&=\sharp\left\{(x_1,\ldots,x_n)\in \{0,\ldots,k\}^n\mid \forall \:1\leq i,j \leq n, 
(i\stackrel{t}{\rightarrow} j)\Longrightarrow (x_i\leq x_j)\right\}.
\end{align*}
Note that $ehr^{cl}_t$ does not depend on the indexation of the vertices of $t$. By the duality principle:
\begin{align*}
(-1)^nehr^{cl}_t(-k)&=\sharp\left\{(x_1,\ldots,x_n)\in \{1,\ldots,k-1\}^n\mid \forall \:1\leq i,j \leq n, (i\stackrel{t}{\rightarrow} j)\Longrightarrow (x_i<x_j)\right\}.
\end{align*}
A B-series is a formal series indexed by rooted trees, of the form:
$$\sum_{t} a_t \frac{t}{aut(t)}=a_{\tun}\tun+a_{\tdeux}\tdeux+a_{\ttroisun}\frac{\ttroisun}{2}+a_{\ttroisdeux}\ttroisdeux+\ldots,$$
where $aut(t)$ is the number of automorphisms of $t$. The following B-series is of special importance in numerical analysis:
$$E=\sum_t \frac{1}{t!}\frac{t}{aut(t)}
=\tun+\frac{1}{2}a_{\tdeux}\tdeux+\frac{1}{3}\frac{\ttroisun}{2}+\frac{1}{6}\ttroisdeux+\ldots,$$
where $t!$ is the tree factorial (see definition \ref{defi33}). This series is the  formal solution of an ordinary differential equation, 
describing the flow of a vector field. The set of B-series is given a group structure by a substitution operation,
which is dually represented by the contraction-extraction coproduct defined in \cite{CEFM}. The inverse of $E$ is called the backward error analysis:
$$E^{-1}=\sum_t \lambda_t \frac{t}{aut(t)!}.$$
Wright and Zhao \cite{Zhao} proved that these coefficients $\lambda_t$ are related to Ehrhart polynomials:
$$\lambda_t=(-1)^{|t|} \frac{d\: ehr^{cl}_t(X)}{dX}_{\mid X=-1}.$$

We shall in this text study Ehrhart polynomial attached to quasi-posets in a combinatorial Hopf-algebraic way. A quasi-poset $P$ is a pair $(A,\leq_P)$,
where $A$ is a finite set and $\leq_P$ is a reflexive and transitive relation on $A$. The isoclasses of quasi-posets are represented by
their Hasse graphs: 
\begin{align*}
&1;&&\tun;&&\tun\tun,\tdeux,\tdun{$2$};&&\tun\tun\tun,\tun\tdeux,\tun\tdun{$2$},\ttroisun,\ptroisun,\ttroisdeux,\tddeux{}{$2$},\tddeux{$2$}{},\tdun{$3$};\ldots
\end{align*}
In particular, rooted trees can be seen as quasi-posets.
For any quasi-poset $P=(\{1,\ldots,n\},\leq_P)$, the polytope associated to $P$ is:
$$pol(P)=\{(x_1,\ldots,x_n)\in [0,1]^n\mid \forall \: 1\leq  i,j\leq n, (i\leq_P j)\Longrightarrow (x_i\leq x_j)\}.$$
We put $ehr_P(X)=ehr^{cl}_{pol(P)}(X-1)$; note the translation by $-1$, which will give us objects more suitable to our purpose.
In other words, for all $k\geq 1$:
\begin{align*}
ehr_P(k)&=\sharp\{(x_1,\ldots,x_n)\in \{1,\ldots,k\}^n\mid \forall \:1\leq i,j \leq n, (i\leq_P j)\Longrightarrow (x_i\leq x_j)\}.
\end{align*}
We also define a polynomial $ehr_P^{str}(X)$ such that for all $k\geq 1$:
\begin{align*}
ehr_P^{str}(k)&=\sharp\{(x_1,\ldots,x_n)\in \{1,\ldots,k\}^n\mid \forall \:1\leq i,j \leq n, (i\leq_P j\mbox{ and not }j\leq_P i)\Longrightarrow (x_i<x_j)\}.
\end{align*}
See definition \ref{defi21} and proposition \ref{prop22} for more details. These polynomials can be inductively computed,
with the help of the minimal elements of $P$ (proposition \ref{prop25}).\\

We shall consider two products $m$ and $\downarrow$, and two coproducts $\Delta$ and $\delta$ on the space $\h_\qp$ generated
by isoclasses of quasi-posets. The coproduct $\Delta$, defined in \cite{FM,FMP} by restriction to open and closed sets
of the topologies associated to quasi-posets, makes $(\h_\qp,m,\Delta)$ a graded, connected Hopf algebra,
and $(\h_\qp,\downarrow,\Delta)$ an infinitesimal bialgebra; the coproduct $\delta$, defined in \cite{FFM} by an extraction-contraction operation, 
makes $(\h_\qp,m,\delta)$ a bialgebra. Moreover, $\delta$ is also a right coaction
of $(\h_\qp,m,\delta)$ over $(\h_\qp,m,\Delta)$, and $(\h_\qp,m,\Delta)$ becomes a Hopf algebra in the category
of $(\h_\qp,m,\delta)$-comodules, which we summarize telling that $(\h_\qp,m,\Delta)$ and $(\h_\qp,m,\delta)$
are two bialgebras in cointeraction (definition \ref{defi1}). For example, the bialgebras $(\K[X],m,\Delta)$ and 
$(\K[X],m,\delta)$ where $m$ is the usual product of $\K[X]$ and $\Delta$, $\delta$ are the coproducts defined by
\begin{align*}
\Delta(X)&=X\otimes 1+1\otimes X,&\delta(X)&=X\otimes X,
\end{align*}
are two cointeracting bialgebras.\\

Ehrhart polynomials $ehr_P(X)$ and $ehr_P^{str}(X)$ can now be seen as maps from $\h_\qp$ to $\K[X]$, and both are Hopf algebra morphisms
from $(\h_\qp,m,\Delta)$ to $(\K[X],m,\Delta)$ (theorem \ref{theo23});
we shall prove in theorem \ref{theo37} that $ehr^{str}$ is the unique morphism from $\h_\qp$ to $\K[X]$ compatible with
both bialgebra structures on $\h_\qp$ and $\K[X]$.
Using the cointeraction between the two bialgebra structures on $\h_\qp$, we show that the monoid $M_\qp$ of characters of $(\h_\qp,m,\delta)$
acts on the set $E_{\h_\qp\rightarrow \K[X]}$ of Hopf algebra morphisms from $(\h_\qp,m,\Delta)$ to $\K[X]$ (lemma \ref{lem5}).
Moreover, there exists a particular homogeneous morphism $\phi_0\in E_{\h_\qp\rightarrow \K[X]}$ such that for all quasi-poset $P$:
$$\phi_0(P)=\lambda_PX^{cl(P)}=\frac{\mu_P}{cl(P)!}X^{cl(P)},$$
where $\mu_P$ is the number of heap-orderings of $P$ and $cl(P)$ is the number of equivalence classes of the equivalence associated to the 
quasi-order of $P$ (proposition \ref{prop32}). This formula simplifies if $P$ is a rooted tree: in this case,
$$\phi_0(P)=\frac{1}{P!}X^{|P|}.$$
We prove that there exist characters $\alpha$ and $\alpha^{str}$ in $M_\qp$, such that for any quasi-poset $P$:
\begin{align*}
ehr_P(X)&=\sum_{\sim\triangleleft P}\frac{\mu_{P/\sim}}{cl(\sim)!} \alpha_{P|\sim}X^{cl(\sim)},&
ehr^{str}_P(X)&=\sum_{\sim\triangleleft P}\frac{\mu_{P/\sim}}{cl(\sim)!} \alpha^{str}_{P|\sim}X^{cl(\sim)},
\end{align*}
where the sum is over a certain family of equivalence relations $\sim$ on the set of vertices of $V$, $P|\sim$ is a restriction operation and $P/\sim$ is a contraction operation.
Applied to corollas, this gives Faulhaber's formula.
We prove that $\alpha^{str}$ is the inverse of the character $\lambda$ associated to $\phi_0$ (theorem \ref{theo37}),
which is a generalization, as well as a Hopf-algebraic proof, of Wright and Zhao's result.
We also give an algebraic proof of the duality principle (theorem \ref{theo38}), and we define a Hopf algebra automorphism 
$\theta:(\h_\qp,m,\Delta)\longrightarrow (h_\qp,m,\Delta)$ with the help of the cointeraction of the two bialgebra structures on $\h_\qp$, 
satisfying $ehr^{str}\circ \theta=ehr$ (proposition \ref{prop39}).\\

We propose non-commutative versions of these results in the last section of the paper. Here, (isoclasses of) quasi-posets are 
replaced by quasi-posets on sets $[n]=\{1,\ldots,n\}$, making a Hopf algebra $\h_\QP$, in cointeraction with $(\h_\qp,m,\delta)$,
and $\K[X]$ is replaced by the Hopf algebra of packed words $\WQSym$ \cite{NovelliThibon}.
We define two surjective Hopf algebra morphisms $EHR$ and $EHR^{str}$ from $\h_\QP$ to $\WQSym$ (proposition \ref{prop41}),
generalizing $ehr$ and $ehr^{str}$. The automorphism $\theta$ is generalized as a Hopf algebra automorphism $\Theta:\h_\QP\longrightarrow \h_\QP$,
such that $EHR^{str}\circ \Theta=EHR$ (proposition \ref{prop42}), and we formulate a non-commutative duality principle
(theorem \ref{theo45}), and we obtain a commutative diagram of Hopf algebras:
$$\xymatrix{\h_\QP\ar@{^(->>}[d]_\Theta \ar@{->>}[rd]^(.6){EHR}\ar@{-->>}@/^1pc/[rrrrdd]&&&&\\
\h_\QP\ar@{^(->>}[d]_\Psi\ar@{->>}[r]^(.4){EHR^{str}}\ar@{-->>}@/^1pc/[rrrrdd]|(.32)\hole&\WQSym\ar@{^(->>}[d]_{\Phi_{-1}}
\ar@{-->>}[rrrrdd]^H|(.7)\hole&&&\\
\h_\QP\ar@{->>}[r]^(.4){EHR}\ar@{-->>}@/^1pc/[rrrrdd]&\WQSym\ar@{-->>}[rrrrdd]^H|(.7)\hole
&&&\h_\qp\ar@{^(->>}[d]_\theta \ar@{->>}[rd]^{ehr}&\\
&&&&\h_\qp\ar@{^(->>}[d]_\psi\ar@{->>}[r]_{ehr^{str}}&\K[X]\ar@{^(->>}[d]_{\phi_{-1}}\\
&&&&\h_\qp\ar@{->>}[r]_{ehr}&\K[X]}$$
The two triangles reflects the properties of morphisms $\Theta$ and $\theta$, whereas the two squares are the duality principles.\\

\textbf{Aknowledgment.} The research leading these results was partially supported by the French National Research Agency under the reference
ANR-12-BS01-0017.\\

\textbf{Notations}. We denote by $\K$ a commutative field of characteristic zero. All the objects (vector spaces, algebra, and so on)
in this text are taken over $\K$.\\

\section{Bialgebras in cointeraction}

We give in this section some general results on bialgebras in cointeractions. They will be used in the sequel for quasi-posets,
leading to Ehrhart polynomials. We shall use them on graphs in order to obtain chromatic polynomials in \cite{Foissychromatic}.

\subsection{Definition}

\begin{defi}\label{defi1}
Let $A$ and $B$ be two bialgebras. We shall say that $A$ and $B$ are in cointeraction if:
\begin{itemize}
\item $B$ coacts on $A$, via a map $\rho:\left\{\begin{array}{rcl}
A&\longrightarrow&A\otimes B \\
a&\longrightarrow&\rho(a)=a_1\otimes a_0.
\end{array}\right.$
\item $A$ is a bialgebra in the category of $B$-comodules, that is to say:
\begin{itemize}
\item $\rho(1_A)=1_A\otimes 1_B$.
\item $m^3_{2,4}\circ (\rho\otimes \rho)\circ \Delta_A=(\Delta_A\otimes Id)\circ \rho$, with:
$$m^3_{2,4}:\left\{\begin{array}{rcl}
A\otimes B\otimes A\otimes B&\longrightarrow&A\otimes A\otimes B\\
a_1\otimes b_1 \otimes a_2 \otimes b_2&\longrightarrow&a_1 \otimes a_2 \otimes b_1 b_2.
\end{array}\right.$$
Equivalently, in Sweedler's notations, for all $a\in A$:
$$(a^{(1)})_1 \otimes (a^{(2)})_1\otimes (a^{(1)})_0 (a^{(2)})_0=(a_1)^{(1)}\otimes (a_1)^{(2)}\otimes a_0.$$
\item For all $a,b\in A$, $\rho(ab)=\rho(a)\rho(b)$.
\item For all $a\in A$, $(\varepsilon_A\otimes Id)\circ \rho(a)=\varepsilon_A(a)1_B$.
\end{itemize}\end{itemize}\end{defi}

Examples of bialgebras in interaction can be found in \cite{CEFM} (for rooted trees) and in \cite{Manchon} (for various families of graphs).
Another example is given by the algebra $\K[X]$, with its usual product $m$, and the two coproducts defined by:
\begin{align*}
\Delta(X)&=X\otimes 1+1\otimes X,&\delta(X)&=X\otimes X.
\end{align*}
The bialgebras $(\K[X],m,\Delta)$ and $(\K[X],m,\delta)$ are in cointeractions, via the coaction $\rho=\delta$.
Identifying $\K[X]\otimes \K[X]$ and $\K[X,Y]$:
\begin{align*}
\Delta(P)(X,Y)&=P(X+Y),&\delta(P)(X,Y)&=P(XY).
\end{align*}

\textbf{Remark.} If $A$ and $B$ are in cointeraction, the coaction of $B$ on $A$ is an algebra morphism.

\begin{prop}
Let $A$ and $B$ be two bialgebras in cointeraction.
We assume that $A$ is a Hopf algebra, with antipode $S$. Then $S$ is a morphism of $B$-comodules,
that is to say:
\begin{align*}
\rho\circ S&=(S\otimes Id)\circ \rho
\end{align*}\end{prop}

\begin{proof} 
We work in the space $End_\K(A,A\otimes B)$. As $A\otimes B$ is an algebra and $A$ is a coalgebra, 
it is an algebra for the convolution product $\circledast$:
$$\forall f,g\in End_\K(A,A\otimes B),\: f\circledast g=m_{A\otimes B}\circ (f\otimes g)\circ \Delta_A.$$
Its unit is:
$$\eta:\left\{\begin{array}{rcl}
A&\longrightarrow&A\otimes B\\
a&\longrightarrow&\varepsilon(a)1_A\otimes 1_B.
\end{array}\right.$$
We consider three elements in this algebra, respectively $\rho$, $F_1=(S\otimes Id)\circ \rho$ and $F_2=\rho \circ S$. Firstly:
\begin{align*}
(F_1\circledast \rho)(a)&=S((a^{(1)})_1)(a^{(2)})_1\otimes (a^{(1)})_0 (a^{(2)})_0\\
&=S((a_1)^{(1)})(a_1)^{(2)}\otimes a_0\\
&=\varepsilon_A(a_1)1_A\otimes a_0\\
&=\varepsilon_A(a)1_A\otimes 1_B\\
&=\eta(a).
\end{align*}
Secondly:
\begin{align*}
(\rho \circledast F_2)(a)&=(a^{(1)})_1S(a^{(2)})_1\otimes (a^{(1)})_0(S(a^{(2)}))_0\\
&=\varepsilon_A(a)(1_A)_1\otimes (1_A)_0\\
&=\varepsilon_A(a)1_A\otimes 1_B\\
&=\eta(a).
\end{align*}
We obtain that $F_1\circledast\rho=\rho \circledast F_2=\eta$, so 
$F_1=F_1\circledast \eta=F_1\circledast \rho \circledast F_2=\eta \circledast F_2=F_2$. \end{proof}

\subsection{Monoids actions}

\begin{prop} 
Let $A$ and $B$ be two bialgebras in cointeraction, through the coaction $\rho$.
We denote by $M_A$ and $M_B$ the monoids of characters of respectively $A$ and $B$. Then $B$ acts on $A$ by monoid endomorphisms,
via the map:
$$\leftarrow:\left\{\begin{array}{rcl}
M_A\times M_B&\longrightarrow&M_A\\
(\phi,\lambda)&\longrightarrow&\phi \leftarrow \lambda=(\phi\otimes \lambda)\circ \rho.
\end{array}\right.$$
\end{prop}

\begin{proof} We denote by $*$ the convolution product of $M_B$ and by $\star$ the convolution product of $M_A$.
As $\rho:A\longrightarrow A\otimes B$ is an algebra morphism, $\leftarrow$ is well-defined.
Let $\phi\in M_A$, $\lambda_1,\lambda_2\in M_B$.
\begin{align*}
(\phi\leftarrow \lambda_1)\leftarrow \lambda_2&=(\phi\otimes \lambda_1\otimes \lambda_2)\circ (\rho \otimes Id)\circ \rho\\
&=(\phi\otimes \lambda_1\otimes \lambda_2)\circ(Id \otimes \Delta_B)\circ \rho\\
&=\phi\leftarrow (\lambda_1*\lambda_2).
\end{align*}
So $\leftarrow$ is an action. Let $\phi_1,\phi_2\in M_A$, $\lambda\in M_B$. For all $a\in A$:
\begin{align*}
((\phi_1\star \phi_2)\leftarrow \lambda)(a)&=(\phi_1\otimes \phi_2\otimes \lambda)\circ (\Delta_A\otimes Id)\circ \rho(a)\\
&=(\phi_1\otimes \phi_2\otimes \lambda)((a_0)^{(1)}\otimes (a_0)^{(2)}\otimes a_1)\\
&=(\phi_1\otimes \phi_2\otimes \lambda)((a^{(1)})_0\otimes (a^{(2)})_0\otimes (a^{(1)})_1 (a^{(2)})_1)\\
&=\phi_1((a^{(1)})_0)\lambda((a^{(1)})_1)\phi_2((a^{(2)})_0)\lambda((a^{(2)})_1)\\
&=(\phi_1\leftarrow \lambda)(a^{(1)})(\phi_2\leftarrow \lambda)(a^{(2)})\\
&=((\phi_1\leftarrow \lambda)\star (\phi_2\leftarrow \lambda))(a).
\end{align*}
So $\leftarrow$ is an action by monoid endomorphisms. \end{proof}\\

\textbf{Example.} We take $A=(\K[X],m,\Delta)$, $B=(\K[X],m,\delta)$ and $\rho=\delta$. We consider the map:
$$ev:\left\{\begin{array}{rcl}
\K&\longrightarrow&\K[X]^*\\
\lambda&\longrightarrow&\left\{\begin{array}{rcl}
\K[X]&\longrightarrow&\K\\
P(X)&\longrightarrow&ev_\lambda(P)=P(\lambda).
\end{array}\right.
\end{array}\right.$$
Then $ev$ is an isomorphism from $(\K,+)$ to $(M_A,\star)$ and from $(\K,.)$ to $(M_B,*)$. Moreover, for all $\lambda,\mu \in \K$:
$$ev_\lambda \leftarrow ev_\mu=ev_{\lambda\mu}.$$

\begin{prop}\label{prop4}
Let $A$ and $B$ be two bialgebras in cointeraction, through the coaction $\rho$. 
\begin{enumerate}
\item Let $H$ be a bialgebra.
We denote by $M_B$ the monoid of characters of $B$ and by $E_{A\rightarrow H}$ the set of bialgebra morphisms from $A$ to $H$.
Then $M_B$ acts on $E_{A\rightarrow H}$ via the map:
$$\leftarrow:\left\{\begin{array}{rcl}
E_{A\rightarrow H}\times M_B&\longrightarrow&E_{A\rightarrow H}\\
(\phi,\lambda)&\longrightarrow&\phi\leftarrow \lambda=(\phi \otimes \lambda)\circ \rho
\end{array}\right.$$
\item Let $H_1$ and $H_2$ be two bialgebras and let $\theta:H_1\longrightarrow H_2$ be a bialgebra morphism. For all $\phi\in E_{A\leftarrow H_1}$,
for all $\lambda \in M_B$, in $E_{A\leftarrow H_2}$:
$$\theta \circ (\phi\leftarrow \lambda)=(\theta \circ \phi)\leftarrow \lambda.$$
\item  if $\lambda,\mu \in M_B$, in $E_{A\rightarrow A}$:
\begin{align*}
(Id \leftarrow \lambda) \circ (Id \leftarrow \mu)&=Id \leftarrow (\lambda *\mu).
\end{align*}
The following map is an injective monoid morphism:
\begin{align*}
&\left\{\begin{array}{rcl}
(M_B,*)&\longrightarrow&(E_{A\rightarrow A},\circ)\\
\lambda&\longrightarrow&Id\leftarrow \lambda.
\end{array}\right.\end{align*}\end{enumerate}\end{prop}

\begin{proof} 1. For all $\phi\in E_{A\leftarrow B}$, $\lambda\in M_B$,
$\phi\leftarrow \lambda:A\longrightarrow H\otimes \K=H$. As $\phi$, $\lambda$ and $\rho$ are algebra morphisms, 
by composition $\phi\leftarrow \lambda$ is an algebra morphism. Let $a\in A$.
\begin{align*}
\Delta_H(\phi\leftarrow \lambda(a))&=\Delta_H(\phi(a_0)\lambda(a_1))\\
&=\lambda(a_1)\Delta_H\circ \phi(a_1)\\
&=\lambda(a_1)\phi(a_0)^{(1)} \otimes \phi(a_0)^{(2)}\\
&=\lambda(a_1)\phi((a_0)^{(1)})\otimes \phi((a_0)^{(2)})\\
&=\lambda((a^{(1)})_1 (a^{(2)})_1) \phi((a^{(1)})_0)\otimes \phi((a^{(2)})_0)\\
&=\lambda((a^{(1)})_1) \lambda((a^{(2)})_1) \phi((a^{(1)})_0)\otimes \phi((a^{(2)})_0)\\
&=\phi((a^{(1)})_0)\lambda((a^{(1)})_1)\otimes \phi((a^{(2)})_0) \lambda((a^{(2)})_1)\\
&=\phi\leftarrow \lambda(a^{(1)})\otimes \phi\leftarrow \lambda(a^{(2)})\\
&=((\phi\leftarrow \lambda)\otimes(\phi\leftarrow \lambda))\circ \Delta_A(a).
\end{align*}
So $\phi\leftarrow \lambda \in E_{A\rightarrow H}$.  

Let $\phi \in E_{A\rightarrow H}$. For all $a\in A$, $\phi\leftarrow  \eta\circ \varepsilon(a)=\phi(a_0)\varepsilon(a_1)=\phi(a)$. 
Let $\lambda,\mu \in M_B$.
\begin{align*}
(\phi\leftarrow \lambda)\leftarrow \mu&=(\phi \otimes \lambda \otimes \mu)\circ (\rho\otimes Id)\circ \rho
=(\phi \otimes \lambda \otimes \mu)\circ (Id\otimes\Delta_B)\circ \rho=\phi\leftarrow (\lambda*\mu).
\end{align*}
So $\leftarrow $ is indeed an action of $M_B$ on $E_{A\rightarrow H}$.\\

2. Let $a\in H$.
\begin{align*}
(\theta \circ \phi)\leftarrow \lambda(a)&=\theta \circ \phi(a_1) \lambda(a_0)=\theta(\phi(a_1)\lambda(a_0))=\theta(\phi\leftarrow \lambda(a))
=\theta\circ (\phi\leftarrow \lambda)(a).
\end{align*}
So $(\theta \circ \phi)\leftarrow \lambda=\theta \circ (\phi\leftarrow \lambda)$. \\

3. Consequently, if $\lambda,\mu \in M_B$, in $E_{A\rightarrow A}$:
$$(Id \leftarrow \lambda) \circ (Id \leftarrow \lambda)=(Id \leftarrow \lambda)\leftarrow \mu)=Id \leftarrow (\lambda *\mu).$$
If $Id \leftarrow \lambda=Id$, then, composing by $\varepsilon'$, we obtain $\varepsilon'*\lambda=\varepsilon'$, so $\lambda=\varepsilon'$. \end{proof}\\

\textbf{Example.} We take $A=(\K[X],m,\Delta)$, $B=(\K[X],m,\delta)$ and $\rho=\delta$. In $E_{A\longrightarrow A}$, for any $\lambda \in \K$:
$$Id\leftarrow ev_\lambda(X)=ev_\lambda(X)X=\lambda X,$$
so for any $P\in \K[X]$, $(Id\leftarrow ev_\lambda)(P)=P(\lambda X)$.
In this case, the monoids $(M_B,*)$ and $(E_{A\rightarrow A},\circ)$ are isomorphic.

\subsection{Polynomial morphisms}

\label{sect1-3}

In this section, we deal with a family $(A,m,\Delta,\delta)$ such that:
\begin{enumerate}
\item $(A,m,\Delta)$ is a graded, connected Hopf algebra. As a graded algebra, it is isomorphic to the symmetric algebra $S(V)$,
where $V$ is a graded subspace of $A$. 
\item $(A,m,\delta)$ is a bialgebra.
\item $(A,m,\Delta)$ and $(A,m,\delta)$ are in cointeraction, through the coaction $\delta$.
\item $V_1=A_1$ has a basis $(g_i)_{i\in I}$ such that:
$$\forall i \in I,\:\delta(g_i)=g_i\otimes g_i.$$
We shall denote by $J$ the set of sequences  $\alpha=(\alpha_i)_{i\in I}$ with a finite support. For all $\alpha \in J$,
we put $g_\alpha=\prod_{i\in I} g_i^{\alpha_i}$. These are group-like elements of $(A,m,\delta)$.
\item For all $n\geq 2$, $V_n$ can be decomposed as:
$$V_n=\bigoplus_{i\in I,\alpha \in J} V_n(g_i,g_\alpha),$$
such that for all $x\in V_n(g_i,g_\alpha)$:
$$\delta(x)-g_i\otimes x-x\otimes g_\alpha \in S(V_1\oplus\ldots \oplus V_{n-1})^{\otimes 2}.$$
\end{enumerate}
The counit of $(A,m,\Delta)$ will be denoted by $\varepsilon$ and the counit of $(A,m,\delta)$ by $\varepsilon'$.
We denote by $M_B$ the monoid of characters of $(A,m,\delta)$.\\

\textbf{Remark.} If $x\in V_n(g_i,g_\alpha)$, as $\varepsilon'(g_i)=\varepsilon'(g_\alpha)=1$, necessarily, $\varepsilon'(x)=0$ and:
$$\delta(x)-g_i\otimes x-x\otimes g_\alpha\in Ker(\varepsilon')^{\otimes 2}.$$

\begin{lemma} \label{lem5}
Let $\lambda \in M_B$. It has an inverse in $M_B$ if, and only if, for all $i\in I$, $\lambda(g_i)\neq 0$.
\end{lemma}

\begin{proof} $\Longrightarrow$. Let $\mu$ be the inverse of $\lambda$ in $M_B$. For all $i\in I$,
$\lambda*\mu(g_i)=\lambda(g_i)\mu(g_i)=\varepsilon'(g_i)=1$, so $\lambda(g_i)\neq 0$. \\

$\Longleftarrow$. We define two characters $\mu,\nu \in M_B$ by inductively definining $\mu_n=\mu_{\mid V_n}$
and $\nu_n=\nu_{\mid V_n}$. For $n=1$, we put $\mu_1(g_i)=\nu_1(g_i)=\lambda(g_i)^{-1}$. 
Let us assume that $g_1,\ldots,g_{n-1}$ are already defined, with $n\geq 2$. If $x\in V_n(g_i,g_\alpha)$, we put:
$$\delta(x)-g_i\otimes x-x\otimes g_\alpha=\sum x'_k\otimes x''_k \in S(V_1\oplus\ldots \oplus V_{n-1})^{\otimes 2}.$$
Hence, for all $k$, $\mu(x'_k)$ and $\nu(x''_k)$ are defined. We put:
\begin{align*}
\mu_n(x)&=\prod_{i\in I}\frac{1}{\lambda(g_i)^{\alpha_i}}
\left(\varepsilon'(x)-\mu(g_i)\lambda(x)-\sum \mu(x'_k)\lambda(x''_k)\right),\\
\nu_n(x)&=\frac{1}{\lambda(g_i)}\left(\varepsilon'(x)-\lambda(x)\nu(g_\alpha)-\sum \lambda(x'_k)\nu(x''_k)\right).
\end{align*}
Consequently, $\mu,\nu \in M_B$ and for all $x\in V$, $\mu*\lambda(x)=\lambda*\nu(x)=\varepsilon'(x)$,
so $\mu*\lambda=\lambda*\nu=\varepsilon'$, and finally $\mu=\mu*(\lambda*\nu)=(\mu*\lambda)*\nu=\nu$,
so $\lambda$ is invertible in $M_B$. \end{proof}

\begin{lemma} \label{lem6}
Let $C$ be a graded, connected Hopf algebra. We denote by $C_+=Ker(\varepsilon_C)$ its augmentation ideal.
Let $\lambda:C_+\longrightarrow\K$ be any linear map. There exists a unique coalgebra morphism
$\phi:C\longrightarrow \K[X]$ such that:
$$\forall x\in C_+,\:\frac{d\phi(x)}{dX}(0)=\lambda(x).$$
Moreover:
\begin{enumerate}
\item $\phi$ is homogeneous if, and only if, for all $n\geq 2$, $\lambda(C_n)=(0)$.
\item $\phi$ is a Hopf algebra morphism if, and only if, for all $x,y\in C_+$, $\lambda(xy)=0$.
\end{enumerate}\end{lemma}

\begin{proof} Let $\pi:\K[X]\longrightarrow Vect(X)$ be the canonical projection. For any $P\in \K[X]$:
$$\pi(P)=\frac{dP}{dX}(0)X.$$

\textit{Existence}. We define $\phi_{\mid C_n}$ by induction on $n$. For $n=0$, we put $\phi(1)=1$.
Let us assume that $\phi_{\mid C_0\oplus \ldots \oplus C_{n-1}}$ is defined such that for all $x\in C_0\oplus \ldots \oplus C_{n-1}$,
\begin{align*}
\Delta\circ \phi(x)&=(\phi\otimes \phi)\circ \Delta(x),&\pi \circ \phi(x)&=\lambda(x)X.
\end{align*}

Let $x\in C_n$. As $\tdelta(x)\in (C_1\oplus \ldots \oplus C_{n-1})^{\otimes 2}$, we can consider the element
$$y=(\phi \otimes \phi)\circ \tdelta(x)\in \K[X]_+^{\otimes 2}.$$
We put:
$$y=\sum_{i,j\geq 1} a_{i,j} \frac{X^i}{i!}\otimes \frac{X^j}{j!}.$$
Moreover:
\begin{align*}
\tdelta\otimes Id(y)&=((\tdelta \circ \phi)\otimes \phi)\circ \tdelta(x)\\
&=(\phi\otimes \phi\otimes \phi)\circ (\tdelta \otimes Id)\circ \tdelta(x)\\
&=(\phi\otimes \phi\otimes \phi)\circ (Id \otimes \tdelta)\circ \tdelta(x)\\
&=(\phi \otimes (\tdelta \circ \phi))\circ \tdelta(x)\\
&=(Id \otimes \tdelta)(y).
\end{align*}
Hence:
\begin{align*}
\sum_{i,j,k\geq 1} a_{i+j,k}\frac{X^i}{i!}\otimes \frac{X^j}{j!}\otimes\frac{X^k}{k!}
&=\sum_{i,j,k\geq 1} a_{ij+k}\frac{X^i}{i!}\otimes \frac{X^j}{j!}\otimes\frac{X^k}{k!}.
\end{align*}
For all $i,j,k\geq 1$, $a_{i+j,k}=a_{i,j+k}$, so there exist scalars $a_n$ such that for all $i,j\geq 1$, $a_{i,j}=a_{i+j}$. We obtain that:
\begin{align*}
y&=\sum_{n\geq 2} a_n \left(\sum_{i,j \geq 1,i+j=n} \frac{X^i}{i!}\otimes \frac{X^j}{j!}\right)=
\sum_{n\geq 2}a_n \tdelta\left(\frac{X^n}{n!}\right).
\end{align*}
We then put $\displaystyle \phi(x)=\sum_{n\geq 2}a_n\frac{X^n}{n!}+\lambda(x)X$.

We obtain in this way a coalgebra morphism such that $\pi\circ \phi(x)=\lambda(x)X$ for all $x\in C$.\\

\textit{Unicity}. Let $\phi,\psi:C\longrightarrow \K[X]$ be coalgebra morphisms such that $\pi\circ \phi=\pi\circ \psi$.
Let us prove that $\phi(x)=\psi(x)$ for all $x\in A_n$, $n\geq 0$, by induction on $n$. For $n=0$,
as the unique group-like element of $\K[X]$ is $1$, $\phi(1)=\psi(1)=1$. Let us assume the result at all ranks $k<n$.
Then, if $x\in A_n$:
\begin{align*}
\tdelta \circ \phi(x)&=(\phi\otimes \phi)\circ \tdelta(x)=(\psi\otimes \psi)\circ \tdelta(x)=\tdelta \circ \psi(x),
\end{align*}
so $\phi(x)-\psi(x)\in Ker(\tdelta)=Vect(X)$. Hence, $0=\pi\circ \phi(x)-\pi\circ \psi(x)=\phi(x)-\psi(x)$.\\

$1.\Longrightarrow$. If $\phi$ is homogeneous, then, for all $n\geq 2$, $\phi(C_n)\subseteq Vect(X^n)$, so:
$$\pi\circ \phi(C_n)=\lambda(C_n)X=(0).$$

$1.\Longleftarrow$. Let us go back to the construction of $\phi$ is the \textit{Existence} part.
If $n\geq 1$, $x\in C_1$, then $\phi(x)=\lambda(x)X$ is homogeneous of degree $1$. 
If $n\geq 2$, then, by homogeneity, $a_{i,j}=0$ if $i+j\neq n$, so $a_k=0$ if $n\neq k$, and
$\displaystyle \phi(x)=a_n\frac{X^n}{n!}$: $\phi$ is homogeneous.\\

$2.\Longrightarrow$. Let $x,y\in C_+$. Then $\phi(x),\phi(y)\in \K[X]_+=X\K[X]$, so $\phi(xy)=\phi(x)\phi(y)\in X^2\K[X]$
and $\pi\circ \phi(xy)=\lambda(xy)X=0$.\\

$2.\Longleftarrow$. Let us consider $\phi_1=m\circ(\phi\otimes \phi)$ and $\phi_2=\phi \circ m$. By composition,
they are both coalgebra morphisms from the graded bialgebra $C\otimes C$ to $\K[X]$. Moreover, if $x,y\in C_+$:
\begin{align*}
\pi\circ \phi_1(1\otimes y)&=\pi\circ \phi(y),&
\pi\circ \phi_2(1\otimes y)&=\pi\circ \phi(y),\\
\pi\circ \phi_1(x\otimes 1)&=\pi\circ \phi(x),&
\pi\circ \phi_2(x\otimes 1)&=\pi\circ \phi(x),\\
\pi \circ \phi_1(x\otimes y)&=\pi(\phi(x)\phi(y))=0,&
\pi\circ \phi_2(x\otimes y)&=\lambda(xy)X=0.
\end{align*}
So, for all $z\in (C\otimes C)_+=(\K1\otimes C_+)\oplus (C_+\otimes \K1)\oplus(C_+\otimes C_+)$,
$\pi \circ \phi_1(z)=\pi \circ \phi_2(z)$. By the \textit{Unicity} part, $\phi_1=\phi_2$, so $\phi$ is an algebra morphism. \end{proof}\\

\textbf{Remark.} If $x\in V_1$, $\phi(x)=\lambda(x)X$.

\begin{theo} \label{theo7}
Under the hypotheses 1--5, there exists a unique homogeneous Hopf algebra morphism $\phi_0:(A,m,\Delta)\longrightarrow(\K[X],m,\Delta)$
such that:
$$\forall x\in A_1,\:\phi_0(x)=\varepsilon'(x)X.$$
Moreover, there exists a unique character $\lambda_0 \in M_B$, invertible in $M_B$, such that:
$$\forall n\geq 0,\: \forall x\in A_n, \:\phi_0(x)=\lambda_0(x)X^n.$$
\end{theo}

\begin{proof}
\textit{Existence.} Let $\lambda \in A^*$ such that $\lambda(x)=\varepsilon'(x)$ if $x\in A_1$ and $\lambda(x)=0$ if $x\in A_n$, $n\geq 2$. 
We denote by $\phi_0$ the unique coalgebra morphism such that $\displaystyle \frac{d\phi_0(x)}{dX}(0)=\lambda(x)$ for all $x\in A_+$.
By the first point of lemma \ref{lem6}, $\phi_0$ is homogeneous. As $A_+^2\subseteq A_{\geq 2}$,
by the second point of lemma \ref{lem6}, $\phi_0$ is an algebra morphism.\\

\textit{Unicity.} If $\phi$ is such a morphism, by the first point of lemma \ref{lem6}, for all $x\in A_n$, $n\geq 2$,
$\pi\circ \phi(x)=0$; hence, for all $x\in A_+$, $\pi\circ \phi_0=\pi\circ \phi$. By the unicity in lemma \ref{lem6}, $\phi=\phi_0$.\\

By homogeneity of $\phi_0$, for all $x\in A_n$, there exists a unique scalar $\lambda_0(x)\in \K$ such that:
$$\phi_0(x)=\lambda_0(x)X^n.$$
If $x\in A_m$, $y\in A_n$, $xy\in A_{m+n}$ and then:
$$\phi(xy)=\lambda_0(xy)X^{n+m}=\phi(x)\phi(y)=\lambda_0(x)\lambda_0(y)X^{n+m},$$
so $\lambda_0 \in M_B$. For all $i\in I$, $\phi(g_i)=\varepsilon'(g_i)X=X$, so $\lambda_0(g_i)=1$.
By lemma \ref{lem5}, $\lambda_0$ is an invertible element of $M_B$. \end{proof}

\begin{theo} \label{theo8}
Under the hypotheses 1--5, the following map is a bijection:
$$\theta:\left\{\begin{array}{rcl}
M_B&\longrightarrow&E_{A\rightarrow \K[X]}\\
\lambda&\longrightarrow&\phi_0\leftarrow\lambda.
\end{array}\right.$$
Moreover, if $\phi=\theta(\lambda)$, with $\lambda \in M_B$, then for all $x\in V$,
$$\lambda(x)=\frac{d\phi(x)}{dX}(0).$$
\end{theo}

\begin{proof} Let $\lambda \in M_B$, and $\phi=\phi_0\leftarrow \lambda$.
For all $i\in I$, $\phi(g_i)=X\lambda(g_i)$, so $\displaystyle \frac{d\phi(g_i)}{dX}(0)=\lambda(g_i)$.
If $n\geq 2$ and $x\in V_n(g_i,g_\alpha)$, we put:
$$\delta(x)=g_i\otimes x+\sum x'_k\otimes x''_k,$$
with for all $k$, $x'_k$ homogeneous of degree $\geq 2$ or homogeneous of degree $1$, with $\varepsilon'(x'_k)=0$.
For all $k$, $\pi\circ \phi_0(x'_k)=0$. We obtain:
$$\pi\circ \phi(x)=\pi\left(\lambda(x)X+\sum \phi_0(x'_k) \lambda(x''_k)\right)=\lambda(x)X.$$

By the unicity in lemma \ref{lem6}, $\phi$ is injective. If $\psi\in E_{A\rightarrow \K[X]}$, we define $\lambda \in M_B$ by:
$$\forall x\in V,\:\lambda(x)=\frac{d\psi(x)}{dX}(0).$$
We put $\phi=\phi_0\leftarrow\lambda$.
Then for all $x\in V$, $\displaystyle \frac{d\phi(x)}{dX}(0)=\frac{d\psi(x)}{dX}(0)=\lambda(x)$.
If $x\in A_+^2$, by the second point of lemma \ref{lem6}, $\displaystyle \frac{d\phi(x)}{dX}(0)=\frac{d\psi(x)}{dX}(0)=0$.
Finally, as $V$ generates $A$, $A_+=V+A_+^2$, and for all $x\in A_+$, $\displaystyle\frac{d\phi(x)}{dX}(0)=\frac{d\psi(x)}{dX}(0)$.
By the unicity in lemma \ref{lem6}, $\phi=\psi$, so $\theta$ is surjective. \end{proof}

\begin{cor}\label{cor9}
Under the hypotheses 1--5, for any  $\mu\in M_B$, there exists a unique Hopf algebra morphism $\phi:A\longrightarrow\K[X]$, such that:
$$\forall x\in A, \:\phi(x)(1)=\mu(1).$$
This morphism is:$\phi_0\leftarrow(\lambda_0^{*-1}*\mu)$.
\end{cor}

\begin{proof}
Let $\phi$ be a Hopf algebra morphism from $A$ to $\K[X]$. By theorem \ref{theo8}, there exists $\lambda \in M_B$
such that $\phi=\phi_0\leftarrow\lambda$. Let $x\in A$. We write $x=\sum x'_k\otimes x''_k$, with, for all $k$,
$x'_k$ homogeneous of degree $n_k$. Then:
\begin{align*}
\phi(x)(1)&=\left(\sum \lambda_0(x'_k)X^{n_k} \lambda(x''_k)\right)_{\mid X=1}=\sum \lambda_0(x'_k)\lambda(x''_k)=\lambda_0*\lambda(x).
\end{align*}
So $\phi$ satisfies the required conditions if, and only if, $\lambda_0*\lambda=\mu$, if, and only if,
$\lambda=\lambda_0^{*-1}*\mu$, as $\lambda_0$ is invertible. So such a $\phi$ exists and is unique. \end{proof}

\begin{cor}\label{cor10}
Under the hypotheses 1--5, there exists a unique morphism $\phi_1:A\longrightarrow \K[X]$ such that:
\begin{enumerate}
\item $\phi_1$ is a Hopf algebra morphism from $(A,m,\Delta)$ to $(\K[X],m,\Delta)$.
\item $\phi_1$ is a bialgebra morphism from $(A,m,\delta)$ to $(\K[X],m,\delta)$.
\end{enumerate}
Moreover, $\phi_1=\phi_0\leftarrow \lambda_0^{*-1}$ and, for all $x\in A$, $\phi_1(x)(1)=\varepsilon'(x)$.
\end{cor}

\begin{proof} \textit{Unicity}. If such a $\phi_1$ exists, then for all $x\in A$, $\phi_1(x)(1)=\varepsilon'\circ \phi_1(x)=\varepsilon'(x)$.
By corollary \ref{cor9}, $\phi_1=\phi_0\leftarrow(\lambda_0^{*-1}*\varepsilon')=\phi_0\leftarrow \lambda_0^{*-1}$.\\

\textit{Existence.} Let $\phi_1$ be the unique Hopf algebra morphism in $E_{A\rightarrow\K[X]}$ such that for all $x\in A$,
$\phi_1(x)(1)=\varepsilon'(x)$. Recall that we identify $\K[X]\otimes \K[X]$ with $\K[X,Y]$; for all $P\in \K[X]$:
\begin{align*}
\Delta(P)(X,Y)&=P(X+Y),&\delta(P)(X,Y)&=P(XY).
\end{align*}
Let us fix $x\in A$. Let us prove that for all $k,l\in \N^*$, $\delta\circ \phi_1(x)(k,l)=(\phi_1\otimes \phi_1)\circ \Delta(x)(k,l)$ by induction on $k$.
First, observe that $\delta \circ \phi_1(x)(k,l)=\phi_1(x)(kl)$. 
If $k=1$, then:
\begin{align*}
(\phi_1\otimes \phi_1)\circ \Delta(x)(1,l)&=\phi_1((\varepsilon'\otimes Id)\circ \Delta(x))(l)=\phi_1(x)(l).
\end{align*}
Let us assume the result at rank $k$. Then:
\begin{align*}
(\phi_1\otimes \phi_1)\circ \Delta(x)(k+1,l)&=(\Delta \otimes Id)\circ (\phi_1 \otimes \phi_1) \circ \delta(x)(k,1,l)\\
&=(\phi_1\otimes \phi_1 \otimes \phi_1)\circ (\Delta \otimes Id)\circ \delta(x)(k,1,l)\\
&=(\phi_1\otimes \phi_1 \otimes \phi_1)\circ m^3_{2,4}\circ (\delta \otimes \delta)\circ \Delta(x)(k,1,l)\\
&=(\phi_1\otimes \phi_1\otimes \phi_1 \otimes \phi_1)\circ(\delta \otimes \delta)\circ \Delta(x)(k,l,1,l)\\
&=(\phi_1 \otimes \phi_1)\circ \Delta(x)(kl,l)\\
&=\Delta \circ \phi_1(x)(kl,l)\\
&=\phi_1(x)(kl+l)\\
&=\phi_1(x)((k+1)l).
\end{align*}
So the result is true for all $k,l\geq 1$. Hence, $(\phi_1 \otimes \phi_1)\circ \delta(x)=\delta \circ \phi_1(x)$: $\phi_1$ is a bialgebra morphism
from $(A,m,\delta)$ to $(\K[X],m,\delta)$.
\end{proof}

\section{Examples from quasi-posets}

\subsection{Definition}

\begin{defi}
\begin{enumerate}
\item Let $A$ be a set finite set. A quasi-order on $A$ is a transitive, reflexive relation $\leq$ on $A$.
If $\leq$ is a quasi-order on $A$, we shall say that $(A,\leq)$ is a quasi-poset. If $P$ is a quasi-poset:
\begin{enumerate}
\item Its isoclass is denoted by $\isoclass{P}$.
\item $\sim_P$ is defined by:
$$\forall a,b\in A,\: a\sim_P b\mbox{ if }(a\leq b \mbox{ and }b\leq a).$$
It is an equivalence on $A$.
\item $\overline{A}=A/\sim_P$ is given an order by:
$$\forall a,b\in A, \:\overline{a}\leq \overline{b}\mbox{ if }a\leq b.$$
The poset $(\overline{A},\leq)$ is denoted by $\overline{P}$. 
\item The cardinality of $\overline{P}$ is denoted by $cl(P)$.
\end{enumerate}
\item Let $n\in \N$.
\begin{enumerate}
\item The set of quasi-posets which underlying set is $[n]=\{1,\ldots,n\}$ is denoted by $\QP(n)$. 
\item The set of posets which underlying set is $[n]$ is denoted by $\P(n)$.
\item The set of isoclasses of quasi-posets of cardinality $n$ is denoted by $\qp(n)$. 
\item The set of isoclasses of quasi-posets of cardinality $n$ is denoted by $\p(n)$. 
\end{enumerate}
We put:
\begin{align*}
\QP&=\bigsqcup_{n\geq 0} \QP(n),&\P&=\bigsqcup_{n\geq 0} \P(n),&\qp&=\bigsqcup_{n\geq 0} \qp(n),&\p&=\bigsqcup_{n\geq 0} \p(n),\\
\h_\QP&=Vect(\QP),&\h_\P&=Vect(\P),&\h_\qp&=Vect(\qp)&\h_\p&=Vect(\p).
\end{align*}\end{enumerate}\end{defi}

As posets are quasi-posets, there are canonical injections from $\h_\P$ into $\h_\QP$ and from $\h_\p$ into $\h_\qp$.
Moreover, the map $P\longrightarrow \overline{P}$ induces surjective maps from $\h_\QP$ to $\h_\P$ 
and from $\h_\qp$ to $\h_\p$, both denoted by $\xi$. 
The map $P\longrightarrow \isoclass{P}$ induces maps $\isoclass{ }:\h_\QP\longrightarrow \h_\qp$
and $\isoclass{ }:\h_\P\longrightarrow \h_\p$. The following diagram commutes:
\begin{align}
\label{EQ1}&\xymatrix{\h_\QP\ar@{->>}[rd]^{\xi}\ar@{->>}[rr]^{\isoclass{}}&&\h_\qp\ar@{->>}[rd]^{\xi}&\\
&\h_\P\ar@{->>}[rr]^(.3){\isoclass{}}&&\h_\p\\
\h_\P\ar[ru]_{Id}\ar@{->>}[rr]^{\isoclass{}}\ar@{^(->}[uu]&&\h_\p\ar[ru]_{Id}\ar@{^(->}|(.5)\hole[uu]&}
\end{align}

We shall represent any element $P$ of $\QP$ by the Hasse graph of $\overline{P}$, indicating on the vertices the
elements of the corresponding equivalence class.
For example, the elements of $\QP(n)$, $n\leq 3$, are:
\begin{align*}
&1;&&\tdun{$1$};&&\tdun{$1$}\tdun{$2$},\tddeux{$1$}{$2$},\tddeux{$2$}{$1$},\tdun{$1,2$}\hspace{3mm};&&
\tdun{$1$}\tdun{$2$}\tdun{$3$},\tdun{$1$}\tddeux{$2$}{$3$},\tdun{$1$}\tddeux{$3$}{$2$},\tdun{$2$}\tddeux{$1$}{$3$},
\tdun{$2$}\tddeux{$3$}{$1$},\tdun{$3$}\tddeux{$1$}{$2$},\tdun{$3$}\tddeux{$2$}{$1$},
\tdun{$1$}\tdun{$2,3$}\hspace{3mm},\tdun{$2$}\tdun{$1,3$}\hspace{3mm},\tdun{$3$}\tdun{$1,2$}\hspace{3mm},
\end{align*}
\begin{align*}
\tdtroisun{$1$}{$3$}{$2$},\tdtroisun{$2$}{$3$}{$1$},\tdtroisun{$3$}{$2$}{$1$},
\pdtroisun{$1$}{$2$}{$3$},\pdtroisun{$2$}{$1$}{$3$},\pdtroisun{$3$}{$1$}{$2$},
\tdtroisdeux{$1$}{$2$}{$3$},\tdtroisdeux{$1$}{$3$}{$2$},\tdtroisdeux{$2$}{$1$}{$3$},\tdtroisdeux{$2$}{$3$}{$1$},
\tdtroisdeux{$3$}{$1$}{$2$},\tdtroisdeux{$3$}{$2$}{$1$},
\tddeux{$1$}{$2,3$}\hspace{3mm},\tddeux{$2$}{$1,3$}\hspace{3mm},\tddeux{$3$}{$1,2$}\hspace{3mm},
\tddeux{$2,3$}{$1$}\hspace{3mm},\tddeux{$1,3$}{$2$}\hspace{3mm},\tddeux{$1,2$}{$3$}\hspace{3mm},
\tdun{$1,2,3$}\hspace{5mm}.
\end{align*}
We shall represent any element $P\in \qp$ by the Hasse graph of $\overline{P}$, indicating on the vertices the cardinality of the corresponding equivalence
class, if this cardinality is not equal to $1$. For example, the elements of $\qp(n)$, $n\leq 3$, are:
\begin{align*}
&1;&&\tun;&&\tun\tun,\tdeux,\tdun{$2$};&&\tun\tun\tun,\tun\tdeux,\tun\tdun{$2$};&&\ttroisun,\ptroisun,\ttroisdeux,\tddeux{}{$2$},\tddeux{$2$}{},\tdun{$3$}.
\end{align*}

\subsection{First coproduct}

By Alexandroff's theorem \cite{Alexandroff,Stong}, 
finite quasi-posets are in bijection with finite topological spaces. Let us recall the definition of the topology attached to a quasi-poset.

\begin{defi} \begin{enumerate}
\item Let $P=(A,\leq)$ be a quasi-poset. An open set of $P$ is a subset $O$ of $A$ such that:
$$\forall i,j\in A,\: (i\in O \mbox{ and }i\leq j)\Longrightarrow (j\in O).$$
The set of open sets of $P$ (the topology associated to $P$) is denoted by $top(P)$.
\item Let $P=(A,\leq)$ be a quasi-poset and $B\subseteq A$. We denote by $P_{\mid B}$ the quasi-poset $(B,\leq_{\mid B})$.
\item Let $P=(A,\leq_P)$ be a quasi-poset. We assume that $A$ is also given a total order $\leq$: for example,
$A$ is a subset of $\N$. If the cardinality of $A$ is $n$, there exists a unique increasing bijection $f$ from $[n]$, with its usual order, to $(A,\leq)$. 
We denote by $Std(P)$ the quasi-poset, element of $\QP(n)$, defined by:
$$\forall i,j \in [n],\: i\leq_{Std(P)} j \Longleftrightarrow f(i)\leq_P f(j).$$
\end{enumerate}\end{defi}

\begin{prop} \begin{enumerate}
\item We define a product $m$ on $\h_\QP$ in the following way: if $P\in \QP(k)$, $Q\in \QP(l)$, then $PQ=m(P,Q) \in \QP(k+l)$ and
\begin{align*}
\forall i,j \in [k+l],\: i\leq_{PQ} \Longleftrightarrow &(1\leq i,j \leq k \mbox{ and }i \leq_P j)\\
&\mbox{ or }(k+1\leq i,j \leq k+l \mbox{ and }i-k\leq_Q j-k).
\end{align*}
\item We define a second product $\downarrow$ on $\h_\QP$ in the following way: 
if $P\in \QP(k)$, $Q\in \QP(l)$, then $PQ=m(P,Q) \in \QP(k+l)$ and
\begin{align*}
\forall i,j \in [k+l],\: i\leq_{PQ} \Longleftrightarrow &(1\leq i,j \leq k \mbox{ and }i \leq_P j)\\
&\mbox{ or }(k+1\leq i,j \leq k+l \mbox{ and }i-k\leq_Q j-k)\\
&\mbox{ or }(1\leq i \leq k<j \leq k+l).
\end{align*}
\item We define a coproduct $\Delta$ on $\h_\QP$ in the following way: 
$$\forall P\in \QP(n), \:\Delta(P)=\sum_{O\in top(P)} Std(P_{\mid [n]\setminus O})\otimes Std(P_{\mid O}).$$
\end{enumerate}
Then $(\h_\QP,m,\Delta)$ is a non-commutative, non-cocommutative Hopf algebra, and $(\h_\QP,\downarrow,\Delta)$ is an infinitesimal bialgebra.
\end{prop}

\begin{proof} See \cite{FM,FMP}. \end{proof}\\

\textbf{Examples.} If $\{a,b\}=\{1,2\}$ and $\{i,j,k\}=\{1,2,3\}$:
\begin{align*}
\Delta(\tdun{$1$})&=\tdun{$1$}\otimes 1+1\otimes \tdun{$1$},\\
\Delta(\tddeux{$a$}{$b$})&=\tddeux{$a$}{$b$}\otimes 1+1\otimes \tddeux{$a$}{$b$}+\tdun{$a$}\otimes \tdun{$b$},\\
\Delta(\tdtroisun{$i$}{$k$}{$j$})&=\tdtroisun{$i$}{$k$}{$j$}\otimes 1+1\otimes \tdtroisun{$i$}{$k$}{$j$}
+\tddeux{$i$}{$j$}\otimes \tdun{$k$}+\tddeux{$i$}{$k$}\otimes \tdun{$j$}+\tdun{$i$}\otimes \tdun{$j$}\tdun{$k$},\\
\Delta(\pdtroisun{$i$}{$j$}{$k$})&=\pdtroisun{$i$}{$j$}{$k$}\otimes1+1\otimes \pdtroisun{$i$}{$j$}{$k$}
+\tdun{$j$}\otimes \tddeux{$k$}{$i$}+\tdun{$k$}\otimes \tddeux{$j$}{$i$}+\tdun{$j$}\tdun{$k$}\otimes \tdun{$i$},\\
\Delta(\tdtroisdeux{$i$}{$j$}{$k$})&=\tdtroisdeux{$i$}{$j$}{$k$}\otimes 1+1\otimes \tdtroisdeux{$i$}{$j$}{$k$}
\tdun{$i$}\otimes \tddeux{$j$}{$k$}+\tddeux{$i$}{$j$}\otimes \tdun{$k$}.
\end{align*}

\textbf{Remark}. This Hopf algebraic structure is compatible with the morphisms of (\ref{EQ1}), that is to say:
\begin{enumerate}
\item $\h_\P$ is a Hopf subalgebra of $\h_\QP$.
\item observe that:
\begin{itemize}
\item If $(P_1,P_2)$ and $(Q_1,Q_2)$ are pairs of isomorphic quasi-posets, then $P_1Q_1$ and $P_2Q_2$ are isomorphic.
\item If $P_1$ and $P_2$ are isomorphic quasi-posets of $\QP(n)$, and if $\phi:[n]\longrightarrow [n]$ is an isomorphism from $P_1$
to $P_2$, then the topology associated to $P_2$ is the image by $\phi$ of the topology associated to $P_1$ and for any subset $I$
of $P_1$, $\phi_{\mid I}$ is an isomorphism from $(P_1)_{\mid I}$ to $(P_2)_{\mid \phi(I)}$.
\end{itemize}
Consequently, the surjective map $\isoclass{}:\h_\QP\longrightarrow \h_\qp$ is compatible with the product and the coproduct:
$\h_\qp$ inherits a Hopf algebra structure. Its product is the disjoint union of quasi-posets.
For any quasi-poset $P=(A,\leq_P)$:
$$\Delta(\isoclass{P})=\sum_{O\in top(P)} \isoclass{P_{\mid A\setminus O}}\otimes\isoclass{P_{\mid O}}.$$
\item $\h_\p$ is a Hopf subalgebra of $\h_\qp$.
\item All the morphisms in (\ref{EQ1}) are Hopf algebra morphisms.
\end{enumerate}

\begin{defi}\begin{enumerate}
\item We shall say that a finite quasi-poset $P=(A,\leq_P)$ is connected if its associated topology is connected.
\item For any finite quasi-poset $P$, we denote by $cc(P)$ the number of connected components of its associated topology.
\end{enumerate}\end{defi}

It is well-known that $P$ is connected if, and only if, the Hasse graph of $\overline{P}$ is connected.
Any quasi-poset $P$ can be decomposed as the disjoint union of its connected components; in an algebraic setting,
$\h_\qp$ is generated as a polynomial algebra by the connected quasi-posets.
This is not true in $\h_\QP$: for example, $\tddeux{$1$}{$3$}\tdun{$2$}$ is both not connected and indecomposable in $\h_\QP$.

\subsection{Second coproduct}

\begin{defi}
Let $P=(A,\leq_P)$ be a quasi-poset and let $\sim$ be an equivalence on $A$.
\begin{enumerate}
\item We define a second quasi-order $\leq_{P|\sim}$ on $A$ by the relation:
$$\forall x,y\in A,\:x\leq_{P|\sim} y\mbox{ if }(x\leq_P y\mbox{ and }x\sim y).$$
\item We define a third quasi-order $\leq_{P/\sim}$ on $A$ as the transitive closure of the relation $\mathcal{R}$ defined by:
$$\forall x,y\in A,\:x \mathcal{R} y\mbox{ if } (x\leq_P y \mbox{ or }x \sim y).$$
\item We shall say that $\sim$ is $P$-compatible and we shall denote $\sim \triangleleft P$ if the two following conditions are satisfied:
\begin{itemize}
\item The restriction of $P$ to any equivalence class of $\sim$ is connected.
\item The equivalences $\sim_{P/\sim}$ and $\sim$ are equal. In other words:
$$\forall x,y\in A, \: (x\leq_{P/\sim} y\mbox{ and }y\leq_{P/\sim} x) \Longrightarrow x\sim y;$$
note that the converse assertion trivially holds.
\end{itemize}\end{enumerate}\end{defi}

\textbf{Remarks.} \begin{enumerate}
\item $P|\sim$ is the disjoint union of the restriction of $\leq_P$ to the equivalence classes of $\sim$.
\item Let $x,y\in P$. Then $x\leq_{P/\sim} y$ if there exist $x_1,x'_1,\ldots,x_k,x'_k \in A$ such that:
$$x\leq_P x_1 \sim x'_1 \leq_P\ldots \leq_P x_k\sim x'_k \leq_P y.$$
\item If $\sim  \triangleleft P$, then:
\begin{enumerate}
\item The equivalence classes of $\sim_{P/\sim}$ are the equivalence classes of $\sim$ and are included in a connected component of $P$. 
This implies that the connected components of $P/\sim$ are the connected components of $P$. Consequently:
\begin{align}
cl(P/\sim)&=cl(\sim),&cc(P/\sim)&=cc(P),
\end{align}
where $cl(\sim)$ is the number of equivalence classes of $\sim$.
\item If $x\sim_P y$ and $x\sim y$, then $x\sim_{P|\sim} y$: the equivalence classes of $\sim_{P|\sim}$ are the equivalence classes
of $\sim_P$; the connected components of $P|\sim$ are the equivalence classes of $\sim$. Consequently:
\begin{align}
cl(P|\sim)&=cl(P),&cc(P|\sim)&=cl(\sim).
\end{align}\end{enumerate}\end{enumerate}

\begin{defi}
Let $P\in \QP$. We shall say that $P$ is \emph{discrete} if $\isoclass{\overline{P}}=\tun^{cl(P)}$.
\end{defi}

In other words, $P$ is discrete if, and only if, $\sim_P=\leq_P$.

\begin{defi}
We define a second coproduct $\delta$ on $\h_\QP$ in the following way: for all $P\in \QP$,
$$\delta(P)=\sum_{\sim \triangleleft P} (P/\sim) \otimes (P\mid \sim).$$
Then $(\h_\QP,m,\delta)$ is a bialgebra. Its counit $\varepsilon'$ is given by:
$$\forall P\in \QP,\: \varepsilon'(P)=\begin{cases}
1\mbox{ if $P$ is discrete},\\
0\mbox{ otherwise}.
\end{cases}$$
\end{defi}

\begin{proof} Firstly, let us prove the compatibility of $\delta$ and $m$. Let $P=(A,\leq_P)$ and $Q=(B,\leq_Q)$ be two elements of $\QP$.
Let $\sim$ be an equivalence relation on $P$. We denote by $\sim'$ and $\sim''$ the restriction of $\sim$ to $P$ and $Q$. Then:
\begin{itemize}
\item If $\sim\triangleleft PQ$, then as the equivalence classes of $\sim$ are connected, they are included in $A$ or in $B$.
Consequently, if $x\in A$ and $y\in B$, $x$ and $y$ are not equivalent for $\sim$. Moreover, $\sim'\triangleleft P$ and $\sim'' \triangleleft Q$, and:
\begin{align*}
PQ|\sim&=(P|\sim')(Q|\sim''),&PQ/\sim&=(P/\sim')(Q/\sim'').
\end{align*}
\item Conversely, if $\sim'\triangleleft P$,$\sim'' \triangleleft Q$ and for all $x\in A$, $y\in B$, $x$ and $y$ not are not $\sim$-equivalent,
then $\sim\triangleleft PQ$.
\end{itemize}
Hence:
\begin{align*}
\delta(PQ)&=\sum_{\sim \triangleleft PQ} (PQ/\sim)\otimes (PQ|\sim)\\
&=\sum_{\sim' \triangleleft P,\sim'' \triangleleft Q} (P/\sim')(Q/\sim'')\otimes (P|\sim')(Q|\sim'')\\
&=\delta(P)\delta(Q).
\end{align*}

Let us now prove the coassociativity of $\delta$. Let $P\in \QP$. 

\textit{First step.} We put:
\begin{align*}
A&=\{(r,r')\mid r\triangleleft P,\: r'\triangleleft P/r\},&
B&=\{(s,s')\mid s\triangleleft P,\: s'\triangleleft P_\mid s\}.
\end{align*}
We consider the maps:
\begin{align*}
F&:\left\{\begin{array}{rcl}
A&\longrightarrow&B\\
(r,r')&\longrightarrow&(r',r),
\end{array}\right.&
G&:\left\{\begin{array}{rcl}
B&\longrightarrow&A\\
(s,s')&\longrightarrow&(s',s).
\end{array}\right.
\end{align*}

$F$ is well-defined: we put $(s,s')=(r',r)$. The equivalence classes of $s$ are the equivalence classes of $r'$, so are $P$-connected.
If $x\sim_{P/s} y$, there exist $x_1,x'_1,\ldots,x_k,s'_k$ and $y_1,y'_1,\ldots,y_l,y_l$ such that:
\begin{align*}
&x \leq_P x_1 r' x'_1 \leq_P\ldots \leq_P x_kr' x'_k \leq_P y,&
&y \leq_P y_1 r' y'_1 \leq_P\ldots \leq_P y_lr' y'_l \leq_P x.
\end{align*}
Hence:
\begin{align*}
&x \leq_{P/r} x_1 r' x'_1 \leq_{P/r}\ldots \leq_{P/r} x_kr' x'_k \leq_{P/r} y,&
&y \leq_{P/r} y_1 r' y'_1 \leq_{P/r}\ldots \leq_{P/r} y_lr' y'_l \leq_{P/r} x.
\end{align*}
So $x\sim_{P/r} y$. As $r'\triangleleft P/r$, $x\sim_P y$: $s\triangleleft P$.

Let us assume that $x s' y$. Then $x r y$, so, as $r\triangleleft y$, there exists a path from $x$ to $y$ in the Hasse graph of $P$,
made of vertices all $r$-equivalent to $x$ and $y$. If $x'$ and $y$' are two elements of this path,
Then $x' r y'$, so $x' \leq_{G/r} y'$ and finally $x' \leq_{(P/r)/r'} y'$. As $r'\triangleleft P/r$, $x' r'y'$, so $x s y$.
So the elements of this path are all $P|s$-equivalent: the equivalence classes of $s'$ are $P|s$-connected.

Let us assume that $x\sim_{(P\mid s)/s'} y$. There exist $x_1,x'_1,\ldots,x_k,x'_k$ and $y_1,y'_1,\ldots,y_l,y'_l$ such that:
\begin{align*}
&x \leq_{P\mid r'} x_1 r x'_1 \leq_{P\mid r'} \ldots \leq_{P\mid r'} x_kr x'_k \leq_{P\mid r'} y,&
&y \leq_{P\mid r'} y_1 r y'_1 \leq_{P\mid r'}\ldots \leq_{P\mid r'} y_lr y'_l \leq_{P\mid r'} x.
\end{align*}
Then:
\begin{align*}
&x \leq_P x_1 r x'_1 \leq_P \ldots \leq_P x_kr x'_k \leq_P y,&
&y \leq_P y_1 r y'_1 \leq_P\ldots \leq_P y_lr y'_l \leq_P x,
\end{align*}
So $x \leq_{P/r} y$ and $y\leq_{P/r} x$. As $r \triangleleft P$, $x r y$, so $x s' y$: we obtain that $s' \triangleleft P\mid s$. \\

$G$ is well-defined: let $(s,s') \in B$ and let us put $G(s,s')=(r,r')$. The equivalence classes of $r$ are $P|s$-connected, so are $P$-connected.
Let us assume that $x \sim_{P/r} y$.  There exist $x_1,x'_1,\ldots,x_k,x'_k$ and $y_1,y'_1,\ldots,y_l,y'_l$ such that:
\begin{align*}
&x \leq_P x_1 s' x'_1 \leq_P\ldots \leq_P x_k s' x'_k \leq_P y,&
&y \leq_P y_1 s' y'_1 \leq_P\ldots \leq_P y_l s' y'_l \leq_P x.
\end{align*}
As the equivalence classes of $s'$ are $P|s$-connected, all this elements are in the same connected component of $P|s$,
so are $s$-equivalent: 
\begin{align*}
&x \leq_{P\mid s} x_1 s' x'_1 \leq_{P\mid s} \ldots \leq_{P\mid s} x_ks' x'_k \leq_{P\mid s} y,&
&y \leq_{P\mid s} y_1 s' y'_1 \leq_{P\mid s}\ldots \leq_{P\mid s} y_ls' y'_l \leq_{P\mid s} x.
\end{align*}
Hence, $x \sim_{(P\mid s)/s'} y$, so as $s'\triangleleft P\mid s$, $x s' y$, so $x r y$: $r\triangleleft P$.

The equivalence classes of $r'$ are the equivalence classes of $s$, so are $P$-connected and therefore $P/r$-connected.
Let us assume that $x \sim_{(P/r)/r'} y$. Note that if $x' s' y'$, then $x'$ and $y'$ are in the same connected component of $P|s$, so $x' s y$.
By the definition of $\leq_{P/s'}$ as a transitive closure, using this observation, we obtain:
\begin{align*}
&x \leq_P x_1 s x'_1 \leq_P\ldots \leq_P x_k s x'_k \leq_P y,&
&y \leq_P y_1 s y'_1 \leq_P\ldots \leq_P y_l s y'_l \leq_P x.
\end{align*}
So $x \sim_{P/s} y$. As $s \triangleleft P$, $x s y$, so $x r'y$: $r'\triangleleft P/r$.\\

Clearly, $F$ and $G$ are inverse bijections.\\

\textit{Second step.} Let $(r,r') \in A$ and let $F(r,r')=(s,s')$. 
Note that if $x r y$, then $x/\sim_{P/r} y$, so $x/\sim_{(P/r)/r'} y$, so $x r' y$ as $r' \triangleleft P/r$. Then:
\begin{align*}
\leq_{(P/r)/r'}&=\mbox{transitive closure of ($(x r' y)$ or $(x\leq_{P/r} y)$)}\\
&=\mbox{transitive closure of ($(x r' y)$ or $(x\leq_P y)$ or $(x \leq_r y)$)}\\
&=\mbox{transitive closure of ($(x r' y)$ or $(x\leq_P y)$)}\\
&=\mbox{transitive closure of ($(x s y)$ or $(x\leq_P y)$)}\\
&=\leq_{P/s}.
\end{align*}
So $P/s=(P/r)/r'$. 
\begin{align*}
\leq_{(P|s)/s'}&=\mbox{transitive closure of ($(x s' y)$ or $(x\leq_{P|s} y)$)}\\
&=\mbox{transitive closure of ($(x r y)$ or $(x\leq_{P|r'} y)$)}\\
&=\mbox{transitive closure of ($(x r y)$ or ($(x\leq_P y)$ and $(x r' y)$))}\\
&=\mbox{transitive closure of (($(x r y)$ or $(x\leq_P y)$) and ($(s r y)$ or $(x r' y)$))}\\
&=\mbox{transitive closure of ($(x \leq_{Pr/r} y)$ and $(s r' y)$)}\\
&=\leq_{(P/r)|r'}.
\end{align*}
So $(P|s)/s'=(P/r)|r'$. For all $x,y$:
\begin{align*}
x\leq_{(P|s)/s'}y&\Longleftrightarrow (x \leq_{P\mid s} y) \mbox{ and } (x s' y)\\
&\Longleftrightarrow (x \leq_P y) \mbox{ and } x s y \mbox{ and } (x s' y)\\
&\Longleftrightarrow (x \leq_P y) \mbox{ and } x r' y \mbox{ and } (x r y)\\
&\Longleftrightarrow (x \leq_P y) \mbox{ and } (x r y)\\
&\Longleftrightarrow x\leq_{P|r} y.
\end{align*}
So $(P|s)|s'=P|r$. Finally:
\begin{align*}
(\delta \otimes Id)\circ \delta(P)&=\sum_{(r,r')\in A} (P/r)/r' \otimes (P/r)|r'\otimes P|r\\
&=\sum_{(s,s')\in B} P/s\otimes (P|s)/s'\otimes (P|s)|s'\\
&=(Id \otimes \delta)\circ \delta(P).
\end{align*}
So $\h_\QP$ is a bialgebra. \\

Let $P$ be a quasi-poset. If $P$ is discrete, then $\delta(P)=P\otimes P$, so $(\varepsilon'\otimes Id)\circ \delta(P)=(Id \otimes \varepsilon')\circ \Delta(P)=P$.
If $P$ is not discrete, there are three types of relations $\sim \triangleleft P$:
\begin{enumerate}
\item The equivalence classes of $\sim$ are the connected components of $P$: in this case, $P|\sim=P$ and $P/\sim=P_1$ is discrete.
\item $\sim=\sim_P$: in this case, $P/\sim=P$ and $P|\sim=P_2$ is discrete.
\item $\sim$ is not one of two preceding relations: in this case, nor $P/\sim$, nor $P/\sim$ is discrete.
\end{enumerate}
So:
$$\delta(P)-P_1\otimes P-P\otimes P_2\in Ker(\varepsilon')\otimes Ker(\varepsilon'),$$
which implies that $(\varepsilon'\otimes Id)\circ \delta(P)=(Id \otimes \varepsilon')\circ \Delta(P)=P$. \end{proof}\\

\textbf{Examples.} If $\{a,b\}=\{1,2\}$ and $\{i,j,k\}=\{1,2,3\}$:
\begin{align*}
\delta(\tdun{$1$})&=\tdun{$1$}\otimes \tdun{$1$},\\
\delta(\tddeux{$a$}{$b$})&=\tddeux{$a$}{$b$}\otimes \tdun{$a$}\tdun{$b$}+\tdun{$a,b$}\hspace{3mm} \otimes\tddeux{$a$}{$b$},\\
\delta(\tdtroisun{$i$}{$k$}{$j$})&=\tdtroisun{$i$}{$k$}{$j$}\otimes \tdun{$i$}\tdun{$j$}\tdun{$k$}
+\tddeux{$i,j$}{$k$}\hspace{2mm}\otimes \tddeux{$i$}{$j$}\tdun{$k$}
+\tddeux{$i,k$}{$j$}\hspace{2mm}\otimes \tddeux{$i$}{$k$}\tdun{$j$}
+\tdun{$i,j,k$}\hspace{5mm}\otimes \tdtroisun{$i$}{$k$}{$j$},\\
\delta(\pdtroisun{$i$}{$j$}{$k$})&=\pdtroisun{$i$}{$j$}{$k$}\otimes \tdun{$i$}\tdun{$j$}\tdun{$k$}
+\tddeux{$k$}{$i,j$}\hspace{2mm}\otimes \tddeux{$j$}{$i$}\tdun{$k$}
+\tddeux{$j$}{$i,k$}\hspace{2mm}\otimes \tddeux{$k$}{$i$}\tdun{$j$}
+\tdun{$i,j,k$}\hspace{5mm}\otimes \pdtroisun{$i$}{$j$}{$k$},\\
\delta(\tdtroisdeux{$i$}{$j$}{$k$})&=\tdtroisdeux{$i$}{$j$}{$k$}\otimes \tdun{$i$}\tdun{$j$}\tdun{$k$}
+\tddeux{$i,j$}{$k$}\hspace{2mm}\otimes \tddeux{$i$}{$j$}\tdun{$k$}
+\tddeux{$i$}{$j,k$}\hspace{2mm}\otimes \tdun{$i$}\tddeux{$j$}{$k$}
+\tdun{$i,j,k$}\hspace{5mm}\otimes \tdtroisdeux{$i$}{$j$}{$k$}.
\end{align*}

\textbf{Remarks.} \begin{enumerate}
\item $\delta$ is the internal coproduct of \cite{FFM}.
\item $(\h_\QP,m,\delta)$ is not a Hopf algebra: for all $n\geq 1$, 
$\delta(\tdun{$n$})=\tdun{$n$}\otimes \tdun{$n$}$, and $\tdun{$n$}$ has no inverse in $\h_\QP$.
\item This coproduct is also compatible with the map $\isoclass{}$, so we obtain a bialgebra structure on $\h_\qp$ with the coproduct
defined by:
$$\delta(\isoclass{P})=\sum_{\sim \triangleleft P} \isoclass{P/\sim}\otimes \isoclass{P|\sim}.$$
\item $\h_\P$ and $\h_\p$ are not stable under $\delta$, as if $P$ is a poset and $\sim\triangleleft P$, $P/\sim$ is not necessarily a poset
(although $P|\sim$ is). However, there is a way to define a coproduct $\overline{\delta}=(\xi\otimes Id)\circ \delta$ on $\h_p$:
$$\forall P\in \P(n),\:\overline{\delta}(\isoclass{P})=\sum_{\sim\triangleleft P}=\overline{\isoclass{P/\sim}}\otimes \isoclass{P|\sim}.$$
$(\h_\p,m,\overline{\delta})$ is a quotient of $(\h_\qp,m,\delta)$ through the map $\xi$.
\end{enumerate}

\subsection{Cointeractions}

\begin{theo}
We consider the map:
\begin{align*}
\rho=(Id \otimes \isoclass{})\circ \delta&:\left\{\begin{array}{rcl}
\h_\QP&\longrightarrow&\h_\QP \otimes \h_\qp\\
P\in \QP&\longrightarrow&\displaystyle \sum_{\sim \triangleleft P} (P/\sim) \otimes \isoclass{P\mid \sim}.
\end{array}\right.
\end{align*}
The bialgebras $(\h_\QP,m,\Delta)$ and $(\h_\qp,m,\delta)$ are in cointeraction via $\rho$.
\end{theo}

\begin{proof} By composition, $\rho$ is an algebra morphism. Let us take $P\in \QP(n)$. We put:
\begin{align*}
A&=\{(r,O)\mid r\triangleleft P, O\in top(P/r)\},&
B&=\{(O,s,s')\mid O\in top(P), s\triangleleft P_{[n]\setminus O}, s'\triangleleft P_{\mid O}\}.
\end{align*}

\textit{First step.} We define a map $F:A \longrightarrow B$, sending $(r,O)$ to $(O,s,s')$, by:
\begin{itemize}
\item $x s y$ if $x r y$ and $x,y$ are in the same connected component of $P_{\mid [n]\setminus O}$.
\item $x s' y$ if $x r y$ and $x,y$ are in the same connected component of $P_{\mid O}$.
\end{itemize}
Let us prove that $F$ is well-defined. Let us take $x, y\in [n]$, with $x\in O$ and $x\leq_P y$. Then $x \leq_{P/r} y$.
as $O$ is an open set of $P/r$, $y\in O$: $O$ is an open set of $P$. By definition, the equivalence classes of $s'$ are 
the intersection of the equivalence classes of $r$ and of the connected components of $O$. As $O$ is a union of equivalence classes of $r$,
they are $P_{\mid O}$-connected. If $x \sim_{P_{\mid O}/s'} y$, then $x \sim_{P/r} y$ and $x$ and $y$ are in the same connected component of $O$.
As $r\triangleleft r$, $x r y$, so $x s' y$: $s'\triangleleft P_{\mid O}$. Similarly, $s\triangleleft P_{\mid [n]\setminus O}$.\\

\textit{Second step.} We define a map $G:B \longrightarrow A$, sending $(O,s,s')$ to $(O,r,r')$, by:
$$x r y \mbox{ if }(x,y\notin O\mbox{ and }x s y)\mbox{ or }(x,y\in O\mbox{ and }x s' y).$$
Let us prove that $G$ is well-defined. Let $x,y\in [N]$, with $x\in O$ and $x\leq_{P/r} y$. There exists $x_1,x'_1,\ldots,x_k,x'_k$ such that:
$$x \leq_P x_1 r x'_1 \leq_P\ldots \leq_P x_k rx'_k \leq_P y.$$
As $O$ is an open set of $P$, $x_1 \in O$; by definition of $r$, $x'_1\in O$. Iterating, we obtain that $x_2,x'_2,\ldots,x_k,x'_k,y\in O$. So $O$ is open in $P/r$.

Let us assume that $x r y$. Then $x,\in O$ or $x,y\notin O$. As $s\triangleleft P_{\mid [n]\setminus O}$ and $P_{\mid O}$,
there exists a path from $x$ to $y$ in the Hasse graph of $P$ formed by elements $s$- or $s'-$ equivalent to $x$ and $y$,
so the equivalence classes of $r$ are $P$-connected. 

Let us assume that $x \sim_{P/r} y$. here exists $x_1,x'_1,\ldots,x_k,s'_k$ and $y_1,y'_1,\ldots,y_l,y_l$ such that:
\begin{align*}
&x \leq_P x_1 r x'_1 \leq_P\ldots \leq_P x_k r x'_k \leq_P y,&
&y \leq_P y_1 r y'_1 \leq_P\ldots \leq_P y_l r y'_l \leq_P x.
\end{align*}
If $x,y\in O$, then all these elements are in $O$, so $x \sim_{P_{\mid O}/s'} y$, and then $x s' y$, so $x r y$.
If $x,y\notin O$, as $O$ is an open set, none of these elements is in $O$, so $x \sim_{P_{\mid [n]\setminus O}/s} y$,
so $x s y$ and finally $x r y$: $r\triangleleft P$.\\

\textit{Third step.} Let $(r,O)\in A$. We put $F(r,O)=(O,s,s')$ and $G(O,s,s')=(\tilde{r},O)$. If $x r y$, as $O$ is an open set of $P/r$,
both $x$ and $y$ are in $O$ or both are not in $O$. Hence, $x s y$ or $x s' y$, so $x \tilde{r} y$. 

If $x \tilde{r} y$, then $x s y$ or $x s' y$, so $x r y$: $\tilde{r}=r$ and $G\circ F=Id_A$.\\

Let $(O,s,s')\in B$. We put $G(O,s,s')=(r,O)$ and $F(r,O)=(O,\tilde{s},\tilde{s}')$. If $x s y$, then $x$ and $y$ are in the same connected component
of $[n]\setminus O$ as $s \triangleleft P_{\mid [n]\setminus O}$ and $x r y$, so $x \tilde{s} y$.
If $x \tilde{s} y$, then $x r y$, so $x s y$: we obtain that $\tilde{s}=s$. Similarly, $\tilde{s}'=s'$, which proves that $F\circ G=Id_B$.\\

We proved that $F$ and $G$ are inverse bijections. Let $(r,O)\in A$ and $(O,s,s')=F(O,r)$. 
\begin{align*}
\leq_{(P/r)_{\mid [n]\setminus O}}&=\mbox{transitive closure of ($x r y$ and $x \leq_P y$) restricted to $[n]\setminus O$}\\
&=\mbox{transitive closure of ($x r y$ and $x \leq_{P_{\mid [n]\setminus O}} y$)}\\
&=\mbox{transitive closure of ($x s y$ and $x \leq_{P_{\mid [n]\setminus O}} y$)}\\
&=\leq_{P_{\mid [n]\setminus [n]}/s}.
\end{align*}
So $(P/r)_{\mid [n]\setminus O}=P_{\mid [n]\setminus O}/s$. Similarly, $(P/r)_{\mid O}=P_{\mid O}/s'$.

Let us now consider $P_{\mid R}$. Its connected components are the equivalence classes of $r$, that is to say the equivalence classes of $s$ and $s'$;
for any such equivalence class $I$, $(P_{\mid R})_{\mid I}=P_{\mid I}$. So $P_{\mid R}$ is the disjoint union of $(P_{\mid [n]\setminus O})_{\mid s}$
and $(P_{\mid O})_{\mid s'}$, and therefore is isomorphic to $Std(P_{\mid [n]\setminus O})_{\mid s})Std((P_{\mid O})_{\mid s'})$,
but not equal, because of the reindexation induced by the standardization. 
Hence, $\isoclass{P_{\mid R}}=\isoclass{(P_{\mid [n]\setminus O})_{\mid s}} \isoclass{(P_{\mid O})_{\mid s'}}$.

Finally:
\begin{align*}
(\Delta \otimes Id)\circ \rho(P)&=\sum_{(r,O)\in A} (G/r)_{\mid [n]\setminus O}\otimes (G/r)_{\mid O} \otimes \isoclass{G_{|r}}\\
&=\sum_{(O,s,s')\in B}(P_{\mid [n]\setminus O})/s\otimes (P_{\mid O})/s'\otimes \isoclass{(P_{\mid [n]\setminus O})_{\mid s}}
\isoclass{(P_{\mid O})_{\mid s'}}\\
&=m^3_{2,4}\circ (\rho \otimes \rho)\circ \Delta(P).
\end{align*}
Moreover, $(\varepsilon\otimes Id)\circ \rho(P)=\delta_{P,1}1\otimes 1=\varepsilon(P)1\otimes 1$. \end{proof}\\

\textbf{Remark.} As noticed in \cite{FFM}, $(\h_\QP,m,\Delta)$ and $(\h_\QP,m,\delta)$ are not in cointeraction through $\delta$.\\

Taking the quotient through $\isoclass{}$:

\begin{cor}
The bialgebras $(\h_\qp,m,\Delta)$ and $(\h_\qp,m,\delta)$ are in cointeraction via $\delta$.
Moreover, $\h_\qp$ si given a graduation by:
$$\forall n\geq 0, (\h_\qp)_n=Vect(P\in \qp\mid cl(P)=n).$$
With this graduation, hypotheses 1--5 of section \ref{sect1-3} are satisfied.
\end{cor}

\begin{proof} Here, $V$ is the space generated by the set of connected quasi-posets; the basis $(g_i)_{i\in I}$ of group-like elements
of $(\h_\qp)_1$ is $(\tdun{$n$})_{n\geq 1}$. \end{proof}\\

We denote by $M_\qp$ the monoid of characters of $(\h_\qp,m,\delta)$. Its product is denoted by $*$.
Using proposition \ref{prop4} on $\h_\QP$:

\begin{cor} \label{cor20}
Let $\lambda\in M_\qp$. The following map is a Hopf algebra endomorphism:
\begin{align*}
\phi_\lambda&:\left\{\begin{array}{rcl}
(\h_\QP,m,\Delta)&\longrightarrow&(\h_\QP,m,\Delta)\\
P\in \QP&\longrightarrow&\displaystyle \sum_{\sim \triangleleft P} \lambda_{\isoclass{P\mid \sim}} P/\sim.
\end{array}\right.
\end{align*}
It is bijective if, and only if, for all $n\geq 1$, $\lambda_{\tdun{$n$}}\neq 0$. If this holds, $\phi_{\lambda}^{-1}=\phi_{\lambda^{*-1}}$.
\end{cor}

\begin{proof} $\phi_\lambda=Id\leftarrow \lambda$, so is an element of $E_{\h_\QP\rightarrow\h_\QP}$.\\

$\Longrightarrow$. For all $n\geq 0$, $\phi_\lambda(\tdun{$n$})=\lambda_{\tdun{$n$}}\tdun{$n$}$.
As $\phi_\lambda$ is injective,  $\lambda_{\tdun{$n$}}\neq 0$.\\

$\Longleftarrow$. By proposition \ref{lem5}, $\lambda$ is invertible in $M_\qp$: let us denote its inverse by $\mu$.
Then, by proposition \ref{prop4}:
\begin{align*}
\phi_\lambda \circ \phi_\mu&=Id\leftarrow(\lambda*\mu)=Id \leftarrow \varepsilon'=Id.
\end{align*}
Similarly, $\phi_\mu\circ \phi_\lambda=Id$. \end{proof}

\section{Ehrhart polynomials}

\textbf{Notations.} For all $k\geq 0$, we denote by $H_k$ the $k$-th Hilbert polynomial:
$$H_k(X)=\frac{X(X-1)\ldots (X-k+1)}{k!}.$$

\subsection{Definition}

\begin{defi} \label{defi21}
Let $P\in \QP(n)$ and $k\geq 1$. We put:
\begin{align*}
L_P(k)&=\{f:[n]\longrightarrow [k]\mid \forall i,j\in [n],i\leq_P j\Longrightarrow f(i)\leq f(j)\},\\
L^{str}_P(k)&=\{f\in L_P(k)\mid \forall i,j\in [n],(i\leq_P j\mbox{ and } f(i)=f(j))\Longrightarrow i\sim_P j\},\\
W_P(k)&=\{w\in L_P(k)\mid w([n])=[k]\},\\
W^{str}_P(k)&=\{w\in L^{str}_P(k)\mid w([n])=[k]\}.
\end{align*}
By convention:
$$L_P(0)=L^{str}_P(0)=W_P(0)=W_P^{str}(0)=\begin{cases}
\emptyset \mbox{ if }P\neq 1,\\
\{1\}\mbox{ if }P=1.
\end{cases}$$
We also put:
\begin{align*}
L_P&=\bigcup_{k\geq 0} L_P(k),&
L^{str}_P&=\bigcup_{k\geq 0} L^{str}_P(k),&
W_P&=\bigsqcup_{k\geq 0} W_P(k),&
W^{str}_P&=\bigsqcup_{k\geq 0} W^{str}_P(k).
\end{align*}
Note that the elements of $W_P$ and $W_P^{str}$ are packed words (see definition \ref{defi40}).
\end{defi}

\begin{prop} \label{prop22}
Let $P\in \QP$. There exist unique polynomials $ehr_P$ and $ehr^{str}_P\in \mathbb{Q}[X]$, such that for $k\geq 0$:
\begin{align*}
ehr_P(k)&=\sharp L_P(k),&
ehr^{str}_P(k)&=\sharp L^{str}_P(k).
\end{align*} \end{prop}

\begin{proof}  This is obvious if $P=1$, with $ehr_1(X)=ehr_1^{str}(X)=1$.
Let us assume that $P\in \QP(n)$, $n\geq 1$. Note that if $i>n$, $W_P(i)=0$. For all $k\geq 1$:
$$\sharp L_P(k)=\sum_{i=1}^k \sharp W_P(i) \binom{k}{i}=\sum_{i=1}^k \sharp W_P(i)H_i(k)=\sum_{i=1}^n \sharp W_P(i)H_i(k).$$
So:
$$ehr_P(X)=\displaystyle \sum_{i=1}^n \sharp W_P(i)H_i(X).$$
Moreover, if $k=0$:
$$ehr_P(0)=\sum_{i=1}^n \sharp W_P(i)H_i(0)=\sharp L_P(0).$$
In the same way:
$$ehr^{str}_P(X)=\displaystyle \sum_{i=1}^n \sharp W^{str}_P(i)H_i(X).$$
These are indeed elements of $\mathbb{Q}[X]$. \end{proof}\\

\textbf{Remarks.} \begin{enumerate}
\item Let $P,Q \in \QP(n)$. 
\begin{itemize}
\item If they are isomorphic, then $ehr_P(k)=ehr_Q(k)$ for all $k\geq 1$, so $ehr_P=ehr_Q$. 
\item If $w\in L_P$, for all $x,y\in P$ such that $x\sim_P y$, then $w(x)\leq w(y)$ and $w(y)\leq w(x)$, so $w(x)=w(y)$:
$w$ goes through the quotient by $\sim_P$. We obtain in this way a bijection from $L_P(k)$ to $L_{\overline{P}}(k)$ for all $k$, 
so $ehr_P=ehr_{\overline{P}}$. Similarly, $ehr^{str}_P=ehr^{str}_{\overline{P}}$.
\end{itemize}
Hence, we obtain maps, all denoted by $ehr$ and $ehr^{str}$, such the following diagrams commute:
\begin{align*}
&\xymatrix{\h_\QP\ar@{->>}^{\isoclass{}}[r] \ar[rd]_{ehr}&\h_\qp\ar@{->>}^{\xi}[r] \ar[d]^{ehr}&\h_\p \ar[ld]^{ehr}\\
&\K[X]}&&
&\xymatrix{\h_\QP\ar@{->>}^{\isoclass{}}[r] \ar[rd]_{ehr^{str}}&\h_\qp\ar@{->>}^{\xi}[r] \ar[d]^(.44){ehr^{str}}&\h_\p \ar[ld]^{ehr^{str}}\\
&\K[X]}&
\end{align*}
\item Let $P \in \P(n)$. The classical definition of the Ehrhart polynomial $ehr^{cl}(t)$ is the number of of integral points of $tPol(P)$, 
where $Pol(P)$ is the polytope associated to $P$. Hence, $ehr^{cl}(X)=ehr(X+1)$. 
\end{enumerate}

\begin{theo} \label{theo23}
The morphisms $ehr,ehr^{str}:\h_\QP,\h_\qp,\h_\p\longrightarrow\K[X]$ are Hopf algebra morphisms. 
\end{theo}

\begin{proof} It is enough to prove it for $ehr,ehr^{str}:\h_\p\longrightarrow \K[X]$.\\

\textit{First step.} Let $P\in \P(n)$. Let us prove that for all $k,l\geq 0$:
\begin{align*}
ehr_P(k+l)&=\sum_{O\in Top(P)} ehr_{P_{\mid [n]\setminus O}}(k) ehr_{P_{\mid O}}(l),&
ehr^{str}_P(k+l)&=\sum_{O\in Top(P)} ehr^{str}_{P_{\mid [n]\setminus O}}(k) ehr^{str}_{P_{\mid O}}(l).
\end{align*}

Let $f\in L_P(k+l)$. We put $O=f^{-1}(\{k+1,\ldots,k+l\})$. If $x\in O$ and $x\leq_P y$, then $f(x)\leq f(y)$, so $y\in O$: $O$ is an open set of $P$. 
By restriction, the following maps are elements of respectively $L_{P_{\mid [n]\setminus O}}(k)$ and $L_{P_{\mid O}}(l)$:
\begin{align*}
f_1&:\left\{\begin{array}{rcl}
[n]\setminus O&\longrightarrow&[k]\\
x&\longrightarrow&f(x),
\end{array}\right.&
f_2&:\left\{\begin{array}{rcl}
O&\longrightarrow&[l]\\
x&\longrightarrow&f(x)-k.
\end{array}\right.
\end{align*}
This defines a map:
$$\upsilon:\left\{\begin{array}{rcl}
L_P(k+l)&\longrightarrow&\displaystyle\bigsqcup_{O\in Top(P)} L_{P_{\mid [n]\setminus O}}(k)\times L_{P_{\mid O}}(l)\\
f&\longrightarrow&(f_1,f_2).
\end{array}\right.$$
This map is clearly injective; moreover:
$$\nu(L^{str}_P(k+l))\subseteq \bigsqcup_{O\in Top(P)} L^{str}_{P_{\mid [n]\setminus O}}(k)\times L^{str}_{P_{\mid O}}(l).$$
Let us prove that $f$ is surjective. Let $(f_1,f_2)\in L_{P_{\mid [n]\setminus O}}(k)\otimes L_{P_{\mid O}}(l)$, with $O\in Top(P)$. We define a map
$f:P\longrightarrow [k+l]$ by:
$$f(x)=\begin{cases}
f_1(x)\mbox{ if }x\notin O,\\
f_2(x)+k\mbox{ if }x\in O.
\end{cases}$$
Let $x\leq_P i$. As $O$ is an open set of $P$, three cases are possible:
\begin{itemize}
\item $x,y\notin O$: then $f_1(x)\leq f_1(y)$, so $f(x)\leq f(y)$.
\item $x,y\in O$: then $f_2(x)\leq f_2(y)$, so $f(x)\leq f(y)$.
\item $x\notin O$, $y\in O$: then $f(x)\leq k<fyj)$.
\end{itemize}
So $f\in L_P(k+l)$, and $\upsilon(f)=(f_1,f_2)$: $\upsilon$ is surjective, and finally bijective.
Moreover, if $f_1\in L^{str}_{[n]\setminus O}(k)$ and $f_2\in L^{str}_{P_{\mid O}}(l)$, then $f=\upsilon^{-1}(f_1,f_2)\in L^{str}_P(k+l)$.
Finally:
\begin{align*}
f(L_P(k+l))&=\displaystyle\bigsqcup_{O \in Top(P)} L_{P_{\mid [n]\setminus O}}(k)\times L_{P_{\mid O}}(l),\\
f(L^{str}_P(k+l))&=\displaystyle\bigsqcup_{O\in Top(P)} L^{str}_{P_{\mid [n]\setminus O}}(k)\times L^{str}_{P_{\mid O}}(l).
\end{align*}
Taking the cardinals, we obtain the announced result.\\

\textit{Second step.} Let $P\in \P(m)$, $Q\in \P(n)$, and $f:[m+n]\longrightarrow [k]$. Then $f\in L_{PQ}(k)$
if, and only if, $f_{\mid [m]} \in L_P(k)$ and $Std(f_{\mid [m+n]\setminus [m]}\in L_Q(k)$. So $ehr_{PQ}(k)=ehr_P(k)ehr_Q(k)$,
and then $ehr_{PQ}(X)=ehr_P(X)ehr_Q(X)$: $ehr$ is an algebra morphism. 

Let $P$ be a finite poset, and $k,l \geq 0$. By the first step:
\begin{align*}
(ehr\otimes ehr)\circ \Delta(P)(k,l)&=\sum_{O\in Top(P)} ehr_{P_{\mid [n]\setminus O}}(k) ehr_{P_{\mid O}}(l)\\
&=ehr_P(k+l)\\
&=\Delta \circ ehr(P)(k,l).
\end{align*}
As this is true for all $k,l\geq 1$, $(ehr \otimes ehr)\circ \Delta(P)=\Delta \circ ehr(P)$. Moreover:
$$\varepsilon\circ ehr(P)=ehr_P(0)=\begin{cases}
1\mbox{ if }P=1,\\
0\mbox{ otherwise},
\end{cases}$$
so $\varepsilon \circ ehr=\varepsilon$. \\

 The proof is similar for $ehr^{str}$. \end{proof}

\subsection{Recursive computation of $ehr$ and $ehr^{str}$}

Let us recall this classical result:

\begin{lemma}
We consider the following maps:
$$L:\left\{\begin{array}{rcl}
\K[X]&\longrightarrow&\K[X]\\
H_k(X)&\longrightarrow&H_{k+1}(X).
\end{array}\right.$$
The map $L$ is injective, and $L(\K[X])=\K[X]_+$. Moreover, for all $P\in \K[X]$, for all $n\geq 0$:
$$L(P)(n+1)=P(0)+\ldots+P(n).$$
\end{lemma}

\begin{proof} Let us consider $P=H_k(X)$. For all $n \geq 0$:
\begin{align*}
H_k(0)+\ldots+H_k(n)&=\binom{0}{k}+\ldots+\binom{n}{k}\\
&=\binom{k}{k}+\ldots+\binom{n}{k}\\
&=\binom{n+1}{k+1}\\
&=H_{k+1}(n+1)\\
&=L(H_k)(n+1).
\end{align*}
By linearity, for any $P\in \K[X]$, $L(P)(n+1)=P(0)+\ldots+P(n)$ for all $n\geq 1$. \end{proof}\\

\begin{prop}\label{prop25}
Let $P\in \P(n)$.
\begin{align*}
ehr_P(X)&=L\left(\sum_{\emptyset \neq O\in Top(P)} ehr_{P_{\mid [n]\setminus O}}(X)\right),\\
ehr^{str}_P(X)&=L\left(\sum_{\mbox{\scriptsize $\emptyset\neq O\in Top(P)$, discrete}} ehr_{P_{\mid [n]\setminus O}}^{str}(X)\right).
\end{align*}\end{prop}

\begin{proof} Let $n\geq 1$. As $L_Q(1)$ is reduced to a singleton for all finite poset $Q$:
\begin{align*}
ehr_P(n+1)&=\sum_{O\in Top(P)} ehr_{P_{\mid [n]\setminus O}}(n) ehr_{P_{\mid O}}(1)\\
&=\sum_{\emptyset \neq O\in Top(P)} ehr_{P_{\mid [n]\setminus O}}(n)+ehr_P(n). 
\end{align*}
We put:
\begin{align*}
Q(X)&=\sum_{\emptyset \neq O\in Top(P)} ehr_{P_{\mid [n]\setminus O}}(X).
\end{align*}
In particular:
\begin{align*}
Q(0)&=\sum_{\emptyset \neq O\in Top(P)} ehr_{P_{\mid [n]\setminus O}}(0)=ehr_\emptyset(0)+0=1=ehr_P(1).
\end{align*}
Then:
\begin{align*}
ehr_P(n+1)&=Q(n)+ehr_P(n)\\
&=Q(n)+Q(n-1)+ehr_P(n-1)\\
&\vdots\\
&=Q(n)+Q(n-1)+\ldots+Q(1)+ehr_P(1)\\
&=Q(n)+\ldots+Q(1)+Q(0)\\
&=L(Q)(n+1).
\end{align*}
So $ehr_P=L(Q)$. \\

For $ehr_P^{str}$, observe that $ehr_Q^{str}(1)=1$ if $Q$ is discrete, and $0$ otherwise, which implies:
\begin{align*}
ehr^{str}_P(n+1)&=\sum_{\mbox{\scriptsize $\emptyset \neq O\in Top(P)$, discrete}} ehr^{str}_{P_{\mid [n]\setminus O}}(n)+ehr^{str}_P(n). 
\end{align*}
The end of the proof is similar. \end{proof}\\

\textbf{Examples}.
\begin{align*}
ehr_{\tun}(X)&=H_1(X)=X,\\
ehr_{\tdeux}(X)&=H_1(X)+H_2(X)=\frac{X(X+1)}{2},\\
ehr_{\ttroisun}(X)=ehr_{\ptroisun}(X)&=H_1(X)+3H_2(X)+2H_3(X)=\frac{X(X+1)(2X+1)}{6},\\
ehr_{\ttroisdeux}(X)&=H_1(X)+2H_2(X)+H_3(X)=\frac{X(X+1)(X+2)}{6};\\ \\
ehr^{str}_{\tun}(X)&=H_1(X)=X,\\
ehr^{str}_{\tdeux}(X)&=H_2(X)=\frac{X(X-1)}{2},\\
ehr^{str}_{\ttroisun}(X)=ehr^{str}_{\ptroisun}(X)&=H_2(X)+2H_3(X)=\frac{X(X-1)(2X-1)}{6},\\
ehr^{str}_{\ttroisdeux}(X)&=H_3(X)=\frac{X(X-1)(X-2)}{6}.\\
\end{align*}

\subsection{Characterization of quasi-posets by packed words}

\begin{lemma}\label{lem26}
Let $P\in \QP(n)$ and let $I_1,\ldots,I_k$ be distinct minimal classes of the poset $\overline{P}$; 
let $w' \in W^{str}_{P_{\mid [n]\setminus (I_1\sqcup \ldots \sqcup I_k)}}$. The following map belongs to $W^{str}_P$:
\begin{align*}
w:&\left\{\begin{array}{rcl}
[n]&\longrightarrow&\N^*\\
x\in I_p, 1\leq p\leq k&\longrightarrow&p\\
x\notin I_1\sqcup \ldots \sqcup I_k&\longrightarrow&w'(x).
\end{array}\right.
\end{align*}\end{lemma}

\begin{proof} Let us assume that $i\leq_P j$. 
\begin{itemize}
\item If $i\in I_p$, as $I_p$ is a minimal class of $\overline{P}$, $j\in I_p$ or $j\notin I_1\sqcup \ldots \sqcup I_k$.
In the first case, $w(i)=w(j)$; in the second case, $w(i)\leq k<w(j)$. If moreover $w(i)=w(j)$, then necessarily $j\in I_p$, so $i\sim_P j$.
\item  If $i\notin I_1\sqcup \ldots \sqcup I_k$, as $i\leq_P j$,
$j\notin I_1\sqcup \ldots \sqcup I_k$, so $i\leq_{P_{\mid [n]\setminus (I_1\sqcup \ldots \sqcup I_k)}} j$ and $w'(i) \leq w'(j)$, so $w'(i)\leq w'(j)$.
If moreover $w(i)=w(j)$, then $w'(i)=w'(j)$, so $i\sim_{P_{\mid [n]\setminus (I_1\sqcup \ldots \sqcup I_k)}} j$ and finally $i\sim_P j$.
\end{itemize}
As a conclusion, $w\in W^{str}_P$. \end{proof}\\

\textbf{Remark.} This lemma implies that $W^{str}_P$ is non-empty for any non-empty quasi-poset $P$.

\begin{prop}
Let $P=([n],\leq_P)$ be a quasi-poset and let $i,j\in [n]$. The following assertions are equivalent:
\begin{enumerate}
\item $i\leq_P j$.
\item For all $w\in L_P$, $w(i)\leq w(j)$.
\item For all $w\in L^{str}_P$, $w(i)\leq w(j)$.
\item For all $w\in W_P$, $w(i)\leq w(j)$.
\item For all $w\in W^{str}_P$, $w(i)\leq w(j)$.
\end{enumerate}\end{prop}

\begin{proof} Obviously:
$$\xymatrix{1.\ar@{=>}[r]&2.\ar@{=>}[r]\ar@{=>}[d]&3.\ar@{=>}[d]\\
&4.\ar@{=>}[r]&5.}$$
It is enough to prove that $5.\Longrightarrow 1$. We proceed by induction on $n$. If $n=1$, there is nothing to prove.
Let us assume the result at all ranks $<n$. Let $i,j\in [n]$, such that we do not have $i\leq_P j$. Let us prove that there
exists $w\in W^{str}_P$, such that $w(i)>w(j)$. There exists a minimal element $k\in [n]$, such that $k\leq_P j$;
let $I$ be the class of $k$ in $\overline{P}$. By hypothesis on $i$, $i$ and $k$ are not equivalent for $\sim_P$, so $i\notin I$. 
If $j\in I$, let us choose an element $w'\in W^{str}_{P_{\mid [n]\setminus I}}$; if $j\notin I$, then 
by the induction hypothesis, there exists $w'\in W^{str}_{P_{\mid [n]\setminus I}}$, such that $w'(i)>w'(j)$.
By lemma \ref{lem26}, the following map is an element of $W^{str}_P$:
$$w:\left\{\begin{array}{rcl}
[n]&\longrightarrow&\mathbb{N}\\
x\in I&\longrightarrow&1\\
x\notin I&\longrightarrow&w'(x)+1
\end{array}\right.$$
If $j\in I$, then $w(j)=1<w(i)$; if $j\notin I$, $w(i)=w'(i)+1>w'(j)+1=w(j)$. In both cases, $w(i)>w(j)$. \end{proof}

\subsection{Link with linear extensions}

Let $P\in \QP(n)$. Linear extensions, as defined in \cite{FM}, are elements of $W_P^{str}$: they are the elements $f\in W_P^{str}$ such that
$$\forall i,j\in [n], f(i)=f(j)\Longleftrightarrow i\sim_P j.$$
It may happens that not all elements of $W_P^{str}$ are linear extensions. For example, if $P=\tdtroisun{$1$}{$3$}{$2$}$,
$W_P^{str}(3)=\{(123),(132),(122)\}$, and $(122)$ is not a linear extension of $P$.
The set of linear extensions of $P$ will be denoted by $E_P$.

\begin{defi} \label{defi28}
Let $w$ and $w'$ be two packed words of the same length $n$. We shall say that $w\leq w'$ if:
$$\forall i,j \in [n], \:w(i)<w(j) \Longrightarrow w'(i)<w'(j).$$
\end{defi}

\begin{prop}
Let $P\in \QP(n)$. Then:
$$W_P=\bigcup_{w\in E_P} \{w'\mid w'\leq w\}.$$
This union may be not disjoint. Moreover, the maximal elements of $W_P$ for the order of definition \ref{defi28} are the elements of $E_P$.
\end{prop}

\begin{proof}
$\subseteq$. Let $w\in W_P$. For all $1 \leq p \leq \max(w)$, we put $I_p=w^{-1}(p)$. Let $f_p$ be a linear extension of $P_{\mid I_p}$.
Let us consider:
$$f:\left\{\begin{array}{rcl}
[n]&\longrightarrow&\mathbb{N}\\
i&\longrightarrow&\max(f_1)+\ldots+\max(f_{p-1})+f_p(i) \mbox{ if }i\in I_p.
\end{array}\right.$$
By construction, if $w(i)<w(j)$, then $f(i)<f(j)$: $w\leq f$. Let us prove that $f\in E_P$. 

If $i\leq_P j$, then as $w\in W_P$, $w(i)\leq w(j)$. If $w(i)=w(j)=p$, then $i\leq_{P_{\mid I_p}} j$, so $f_p(i)\leq f_p(j)$,
and $f(i)\leq f(j)$. If $w(i)<w(j)$, then $f(i)<f(j)$. 

If $f(i)=f(j)$, then $w(i)=w(j)=p$, and $f_p(i)=f_p(j)$. As $f_p \in E_{P_{\mid P_p}}$, $i\sim_{P_{\mid P_p}} j$, so $i\sim_P j$.\\

$\supsetneq$. Let $w\in E_P$ and $w'\leq w$. If $i\leq_P j$, then $w(i)\leq w(j)$ as $w$ is a linear extension of $P$. As $w'\leq w$,
$w'(i)\leq w'(j)$, so $w'\in W_P$.\\

Let $w$ be a maximal element of $W_P$. There exists a linear extension $w'$ of $P$, such that $w\leq w'$. As $w$ is maximal, $w=w'$
is a linear extension of $P$. Conversely, if $w$ is a linear extension of $P$ and $w\leq w'$ in $W_P$, then as $w$ is a linear extension of $P$,
$\max(w)=cl(P)$. Moreover, as $w\leq w'$, $\max(w)\leq \max(w')$. As $w'\in W_P$, $\max(w')\leq cl(P)$, which implies
that $\max(w)=\max(w')=cl(P)$, and finally $w=w'$: $w$ is a maximal element of $W_P$. \end{proof} \\

\textbf{Example.} For $P=\tdtroisun{$1$}{$3$}{$2$}$:
\begin{align*}
E_P&=\{(123),(132)\};\\
W_P&=\{(123),(122),(112),(111)\}\cup\{(132),(122),(121),(111)\}\\
&=\{(123),(132),(122),(112),(121),(111)\}.
\end{align*}
Note that the two components of $W_P$ are not disjoint.\\

\textbf{Remark.} A similar result is proved in \cite{FM} for $T$-partitions of a quasi-poset, generalizing Stanley's result \cite{Stanley} for $P$-partitions
of posets; nevertheless, this is different here, as the union is not a disjoint one.

\section{Characters associated to $ehr$ and $ehr^{str}$}

Recall that $(M_\qp,*)$ the monoid of characters of $(\h_\qp,m,\delta)$.\\

By theorems \ref{theo7} and \ref{theo23}:

\begin{prop} 
\begin{enumerate}
\item There exists a unique homogeneous Hopf algebra morphism $\phi_0$ from $(\h_\qp,m,\Delta)$ to $(\K[X],m,\Delta)$ such that:
$$\forall n\geq 1, \:\phi_0(\tdun{$n$})=X.$$
There exists a unique character $\lambda \in M_\qp$ such that for all $P\in \qp$,
$$\phi_0(P)=\lambda_PX^{cl(P)}.$$
\item  There exist unique characters $\alpha$, $\alpha^{str}\in M_\qp$,
such that:
\begin{align*}
ehr&=\phi_0\leftarrow \alpha,&ehr^{str}&=\phi_0\leftarrow \alpha^{str}.
\end{align*}
\end{enumerate}\end{prop}

\textbf{Remarks.}
\begin{enumerate}
\item Let $P\in \qp$. Then $ehr_P=ehr_{\overline{P}}$ and  $ehr^{str}_P=ehr^{str}_{\overline{P}}$,
so $\alpha_P=\alpha_{\overline{P}}$ and  $\alpha^{str}_P=\alpha^{str}_{\overline{P}}$.
\item Still by theorem \ref{theo7}, there exists a unique homogeneous Hopf algebra morphism $\phi'_0$ from $\h_\p$ to $\K[X]$
such that $\phi_0'(\tun)=X$. By unicity, $\phi_0=\phi'_0\circ \Xi$, so for any $P\in \qp$, $\lambda_P=\lambda_{\overline{P}}$.
\end{enumerate}

\subsection{The character $\lambda$}

\begin{lemma}
For all $P\in \P(n)$, $n\geq 0$:
$$\lambda_{\isoclass{P}}=\begin{cases}
1\mbox{ if }P=1,\\
\displaystyle \frac{1}{n}\sum_{M\in max(P)} \lambda_{\isoclass{P_{\mid [n]\setminus \{M\}}}}
=\frac{1}{n}\sum_{m\in min(P)} \lambda_{\isoclass{P_{\mid [n]\setminus \{m\}}}}\mbox{ otherwise}.
\end{cases}$$
\end{lemma}

\begin{proof} Let $P\in \P(n)$, with $n \geq 0$. Recall that $\pi:\K[X]\longrightarrow Vect(X)$ is the canonical projection.
\begin{align*}
(Id \otimes \pi)\circ \Delta\circ \phi_0(\isoclass{P})&=\lambda_{\isoclass{P}} (Id \otimes \pi)\circ \Delta(X^n)\\
&=\lambda_{\isoclass{P}} nX^{n-1};\\
=(Id \otimes \pi)\circ (\phi_0\otimes \phi_0)\circ \Delta(\isoclass{P})
&=\sum_{O \in Top(P)} \lambda_{\isoclass{P_{\mid O}}} \lambda_{\isoclass{P_{\mid [n]\setminus O}}}X^{|[n]\setminus O|}\pi(X^{|O|})\\
&=\sum_{O \in Top(P),\:|O|=1} \lambda_{\isoclass{P_{\mid O}}} \lambda_{\isoclass{P_{\mid [n]\setminus O}}}X^{n-1}\\
&=\sum_{M\in max(P)}\lambda_{\isoclass{P_{\mid [n]\setminus\{M\}}}}X^{n-1}.
\end{align*}
This implies the first equality. The second is proved by considering $(\pi\otimes Id)\circ \Delta\circ \phi_0(\isoclass{P})$. \end{proof} \\

\textbf{Remarks.} \begin{enumerate}
\item This lemma allows to inductively compute $\lambda_P$. This gives:
$$\begin{array}{c|c|c|c|c|c|c|c|c|c|c|c|c|c|c|c}
P&\tun&\tdeux&\ttroisun&\ptroisun&\ttroisdeux&
\tquatreun&\pquatreun&\tquatredeux&\pquatredeux&\tquatrequatre&\pquatrequatre&\tquatrecinq
&\pquatresix&\pquatresept&\pquatrehuit\\ \hline&&&&&&&&&&&&&&&\\[-2mm]
\lambda_P&1&\displaystyle \frac{1}{2}&\displaystyle \frac{1}{3}&\displaystyle \frac{1}{3}&\displaystyle \frac{1}{6}&
\displaystyle \frac{1}{4}&\displaystyle \frac{1}{4}&\displaystyle \frac{1}{8}&\displaystyle \frac{1}{8}&\displaystyle \frac{1}{12}
&\displaystyle \frac{1}{12}&\displaystyle \frac{1}{24}&\displaystyle \frac{5}{24}&\displaystyle \frac{1}{6}&\displaystyle \frac{1}{12}
\end{array}$$
\item If $P=(P,\leq)$ is a finite poset, we denote by $P^{op}$ the opposite poset $(P,\geq)$. It is not difficult to deduce from this lemma that
$\lambda_P=\lambda_{P^{op}}$.
\end{enumerate}

\begin{prop} \label{prop32}
Let $P\in \P(n)$. The number of elements of $W_P(n)$ of $P$ is denoted by $\mu_P$: 
in other words, $\mu_P$ is the number of bijections $f$ from $[n]$ to $[n]$ such that for all $x,y\in [n]$, 
$$x\leq_P y\Longrightarrow f(x)\leq f(y).$$
These bijections are called heap-orderings of $P$. Then, for any finite poset $P$, $\displaystyle \lambda_P=\frac{\mu_P}{n!}$.
\end{prop}

\begin{proof}  Let us fix a non-empty finite poset $P\in \P(n)$. The set of heap-orderings of $P$ is $HO(P)=W_P(n)$.
We consider the map:
$$\left\{\begin{array}{rcl}
HO(P)&\longrightarrow&\displaystyle \bigsqcup_{M\in max(P)} HO(P\setminus \{M\})\\
f&\longrightarrow&f_{\mid [n-1]} \in HO(P\setminus \{f^{-1}(n)\}).
\end{array}\right.$$
It is not difficult to prove that this is a bijection. So:
\begin{align*}
\mu_P&=\sum_{M\in max(P)} \mu_{P\setminus\{M\}};&\frac{\mu_P}{n!}&=\frac{1}{n}\sum_{M\in max(P)} \frac{\mu_{P\setminus\{M\}}}{|P\setminus\{M\}|!}.
\end{align*}
An easy induction on $|P|$ then proves that $\displaystyle \lambda_P=\frac{\mu_P}{n!}$ for all $P$.\end{proof}\\

This formula can be simplified for rooted forests.

\begin{defi}\label{defi33}
Let $P$ be a non-empty finite poset.
\begin{enumerate}
\item  We put:
$$P!=\prod_{i\in V(P)} \sharp\{j\in V(P)\mid i\leq_Pj\}.$$
By convention, $1!=1$.
\item We shall say that $P$ is a rooted forest if $P$ does not contain any subposet isomorphic to $\ptroisun$.
\end{enumerate}\end{defi}

For example, here are isoclasses of rooted forests of cardinality $\leq 4$:
\begin{align*}
&1;&&\tun;&&\tdeux,\tun\tun;&&\ttroisun,\ttroisdeux,\tun\tdeux,\tun\tun\tun;&&
\tquatreun,\tquatredeux,\tquatrequatre,\tquatrecinq,\ttroisun\tun,\ttroisdeux\tun,\tdeux\tdeux,\tdeux\tun\tun,\tun\tun\tun\tun.
\end{align*}

\textbf{Examples.}
$$\begin{array}{c|c|c|c|c|c|c|c|c|c|c|c|c|c|c|c}
P&\tun&\tdeux&\ttroisun&\ptroisun&\ttroisdeux&
\tquatreun&\pquatreun&\tquatredeux&\pquatredeux&\tquatrequatre&\pquatrequatre&\tquatrecinq
&\pquatresix&\pquatresept&\pquatrehuit\\ \hline
P!&1&2&3&4&6&4&8&8&12&12&18&24&6&9&16
\end{array}$$

\begin{prop} \label{prop34}
For all finite poset $P$, $\displaystyle \lambda_P\geq \frac{1}{P!}$, with equality if, and only if, $P$ is a rooted forest.
\end{prop}

\begin{proof} We proceed by induction on $n=|P|$. It is obvious if $n=0$. Let us assume the result at all ranks $<n$.
\begin{align*}
\lambda_P&=\frac{1}{|P|}\sum_{m\in min(P)} \lambda_{P\setminus\{m\}}\\
&\geq \frac{1}{|P|}\sum_{m\in min(P)} \prod_{i\in V(P), i\neq m} \frac{1}{\sharp\{j\in V(P)\mid j\neq m, i\leq_Pj\}}\\
&=\frac{1}{|P|}\sum_{m\in min(P)} \prod_{i\in V(P), i\neq m} \frac{1}{\sharp\{j\in V(P)\mid i\leq_Pj\}}\\
&=\frac{1}{P!}\frac{1}{|P|} \sum_{m\in min(P)} \sharp\{j\in V(P)\mid m\leq_P j\}.
\end{align*}
For any $j\in A$, there exists $m\in min(P)$ such that $m\leq_P j$, so:
$$ \sum_{m\in min(P)} \sharp\{j\in V(P)\mid m\leq_P j\}\geq |P|.$$
Consequently, $\displaystyle \lambda_P\geq \frac{1}{P!}$.\\

Let us assume that this is an equality. Then:
$$ \sum_{m\in min(P)} \sharp\{j\in V(P)\mid m\leq_P j\}=|P|.$$
Consequently, for all $j\in min(P)$, there exists a unique $m\in min(P)$ such that $m\leq_P j$.
Moreover, for all $m\in min(P)$, $\lambda_{P\setminus\{m\}}=\frac{1}{P\setminus\{m\}!}$. By the induction hypothesis,
$P\setminus \{m\}$ is a rooted forest; this implies that $P$ is also a rooted forest.\\

Let us assume that $P$ is a rooted forest. For any $j\in V(P)$, there exists a unique $m \in min(P)$ such that $m\leq_P j$, so:
$$ \sum_{m\in min(P)} \sharp\{j\in V(P)\mid m\leq_P j\}=|P|.$$
Moreover, for all $m\in min(P)$, $P\setminus \{m\}$ is also a rooted forest. By the induction hypothesis,
$\lambda_{P\setminus\{m\}}=\frac{1}{P\setminus\{m\}!}$. Hence, $\displaystyle \lambda_P=\frac{1}{P!}$. \end{proof}

\subsection{The character $\alpha^{str}$}

Let us now apply theorem \ref{theo8} to $ehr$ and $ehr^{str}$:

\begin{theo} \label{theo35}
For all finite connected quasi-poset $P$, we have:
\begin{align*}
\alpha_P&=\frac{d\:ehr_P}{dX}(0),&\alpha^{str}_P&=\frac{d\:ehr^{str}_P}{dX}(0).
\end{align*}
For any quasi-poset $P$:
\begin{align*}
ehr_P(X)&=\sum_{\sim\triangleleft P}\frac{\mu_{P/\sim}}{cl(\sim)!} \alpha_{P|\sim}X^{cl(\sim)},&
ehr^{str}_P(X)&=\sum_{\sim\triangleleft P}\frac{\mu_{P/\sim}}{cl(\sim)!} \alpha^{str}_{P|\sim}X^{cl(\sim)},
\end{align*}
where $cl(\sim)$ is the number of equivalence classes of $\sim$. \end{theo}

Let us give a few values of $\alpha$:
$$\begin{array}{c|c|c|c|c|c|c|c|c|c|c|c|c|c|c|c}
P&\tun&\tdeux&\ttroisun&\ptroisun&\ttroisdeux&
\tquatreun&\pquatreun&\tquatredeux&\pquatredeux&\tquatrequatre&\pquatrequatre&\tquatrecinq
&\pquatresix&\pquatresept&\pquatrehuit\\ \hline&&&&&&&&&&&&&&&\\[-2mm]
\alpha_P&1&\displaystyle \frac{1}{2}&\displaystyle \frac{1}{6}&\displaystyle \frac{1}{6}&\displaystyle \frac{1}{3}&
0&0&\displaystyle \frac{1}{12}&\displaystyle \frac{1}{12}&\displaystyle \frac{1}{6}&\displaystyle \frac{1}{6}
&\displaystyle \frac{1}{4}&\displaystyle \frac{1}{12}&\displaystyle \frac{1}{6}&\displaystyle \frac{1}{6}
\end{array}$$

\begin{lemma}
Let $P\in \QP$, not discrete. Then:
$$(-1)^{cl(P)}ehr_P(-1)=ehr^{str}_P(1)=\begin{cases}
1\mbox{ if $P$ is discrete},\\
0\mbox{ otherwise}.
\end{cases}$$ 
\end{lemma}

\begin{proof} If $P$ is discrete, then $ehr_P(X)=ehr_P^{str}(X)=X^{cl(P)}$ and the result is obvious.
Let us assume that $P$ is not discrete. There exists a unique map $f$ from $[n]$ to $[1]$; as $P$ is not discrete, $f\notin L^{str}_P(1)$, 
so $ehr^{str}_P(1)=0$. We now proceed with $ehr_P(-1)$.\\

\textit{First step.} Let us prove that $L(H_k(-X))=-H_{k+1}(-X)$ for all $k\geq 0$. For all $l,n\geq 0$:
$$H_l(-n)=(-1)^l \binom{n+l-1}{l}.$$
For all $k,n \geq 0$:
\begin{align*}
L(H_k(-X))(n+1)&=H_k(0)+\ldots+H_k(-n)\\
&=(-1)^k\sum_{i=0}^n \binom{i+k-1}{k}\\
&=(-1)^k \sum_{j=k}^{n+k-1}\binom{j}{k}\\
&=(-1)^k\binom{n+k}{k+1}\\
&=-H_{k+1}(-(n+1)).
\end{align*}

\textit{Second step.} Let us prove that $L((X+1)\K[X])\subseteq (X+1)\K[X]$.  For all $k\geq 2$, let us put
$H_k(-X)=X(X+1)L_k(X)$; $(L_k(X))_{k\geq 2}$ is a basis of $\K|X]$, which implies that
$(H_k(-X))_{k\geq 2}$ is a basis of $X(X+1)\K[X]$, and that $(X+1)\sqcup (H_k(-X))_{k\geq 2}$ is a basis of $(X+1)\K[X]$.
First:
$$L(X+1)=L(H_1(X)+H_0(X))=H_2(X)+H_1(X)=\frac{X(X-1)}{2}+X=\frac{X(X+1)}{2}\in (X+1)\K[X];$$
if $k\geq 2$, by the first step, $L(H_k(-X))=-H_{k+1}(-X) \in (X+1)\K[X]$. \\

\textit{Last step.} We can replace $P$ by $\overline{P}$, and we now assume that $P\in \P(n)$. 
There is nothing to prove if $n=0,1$. Let us assume the result at all rank $<n$. Then, by the second step and the induction hypothesis:
\begin{align*}
ehr_{\isoclass{P}}(-1)&=L\left(\sum_{\emptyset \neq O\in Top(P)} ehr_{\isoclass{P_{\mid [n]\setminus O}}}(X)\right)_{\mid X=-1}\\
&=L\left(\sum_{\substack{\mbox{\scriptsize $\emptyset \neq O\in Top(P)$}\\
\mbox{\scriptsize $P_{\mid [n]\setminus O}$ discrete}}} ehr_{\isoclass{P_{\mid [n]\setminus O}}}(X)\right)_{\mid X=-1}\\
&=L\left(\sum_{[n]\neq J\subseteq min(P)} ehr_{\isoclass{P_{\mid J}}}(X)\right)_{\mid X=-1}\\
&=L\left(\sum_{J\subseteq min(P)} ehr_{\isoclass{P_{\mid J}}}J(X)\right)_{\mid X=-1}\\
&=L\left(\sum_{J\subseteq min(P)} X^{|J|}\right)_{\mid X=-1}\\
&=L(\underbrace{(1+X)^{|min(P)|}}_{\in (X+1)\K[X]})_{\mid X=-1}\\
&=0.
\end{align*}
For the fourth equality, note that $P$ is not discrete, so $min(P)\neq P$. \end{proof}\\

In other words, for any $P\in \qp$, $(-1)^{cl(P)}ehr_P(-1)=ehr^{str}_P(1)=\varepsilon'(P)$.
By corollary \ref{cor10}:

\begin{theo} \label{theo37}
$ehr^{str}$ is the unique morphism from $\h_\qp$ to $\K[X]$ such that:
\begin{enumerate}
\item $ehr^{str}$ is a Hopf algebra morphism from $(\h_\qp,m,\Delta)$ to $(\K[X],m,\Delta)$.
\item $ehr^{str}$ is a bialgebra morphism from $(\h_\qp,m,\delta)$ to $(\K[X],m,\delta)$.
\end{enumerate}
Moreover, the character $\alpha^{str}$ is the inverse of $\lambda$ in $M_b$.
\end{theo}

\subsection{The character $\alpha$ and the duality principle}

\begin{theo} \label{theo38}
\begin{enumerate}
\item (Duality principle). For any quasi-poset $P$:
$$ehr^{str}_P(X)=(-1)^{cl(P)}ehr_P(-X).$$
\item For any quasi-poset $P$, $\alpha_P=(-1)^{cl(P)+cc(P)} \alpha^{str}_P$.
\item $\alpha$ is invertible in $M_\qp$. We denote by $\beta$ its inverse. For any quasi-poset $P$:
$$\beta_P=(-1)^{cl(P)+cc(P)}\lambda_P=(-1)^{cl(P)+cc(P)}\frac{\mu_P}{cl(P)!}.$$
\end{enumerate}\end{theo}

\begin{proof} 1. We consider the morphism:
$$\phi:\left\{\begin{array}{rcl}
\h_\qp&\longrightarrow&\K[X]\\
P\in \qp&\longrightarrow&(-1)^{cl(P)}ehr_P(-X).
\end{array}\right.$$
We put:
\begin{align*}
\theta_1&:\left\{\begin{array}{rcl}
\h_\qp&\longrightarrow&\h_\qp\\
P\in \qp&\longrightarrow&(-1)^{cl(P)}P,
\end{array}\right.&
\theta_2&:\left\{\begin{array}{rcl}
\K[X]&\longrightarrow&\K[X]\\
P(X)&\longrightarrow&P(-X).
\end{array}\right.
\end{align*}
Both are Hopf algebra morphisms, and $\phi=\theta_2\circ ehr\circ \phi_1$, so $\phi$ is a Hopf algebra morphism.
If $P$ is a non discrete quasi-poset, then $\phi(P)(1)=(-1)^{cl(P)}ehr_P(-1)=0=\varepsilon'(x)$.
If $P$ is a discrete quasi-poset, then $ehr_P(X)=X^{cl(P)}$, so $\phi(P)(1)=1=\varepsilon'(x)$. By corollary \ref{cor10},
$\phi=\phi_1=ehr^{str}$. \\

2. and 3. Immediate consequences of the first point. \end{proof}

\begin{prop} \label{prop39}
The following map is a Hopf algebra automorphism:
\begin{align*}
\theta&:\left\{\begin{array}{rcl}
(\h_\qp,m,\Delta)&\longrightarrow&(\h_\qp,m,\Delta)\\
P&\longrightarrow&\displaystyle \sum_{\sim\triangleleft P} P/\sim.
\end{array}\right.
\end{align*}
Its inverse is:
\begin{align*}
\theta^{-1}&:\left\{\begin{array}{rcl}
(\h_\qp,m,\Delta)&\longrightarrow&(\h_\qp,m,\Delta)\\
P&\longrightarrow&\displaystyle \sum_{\sim \triangleleft P}(-1)^{cl(\sim)+cl(P)}P/\sim.
\end{array}\right.
\end{align*}
Moreover:
\begin{align*}
ehr^{str}\circ \theta&=ehr,&ehr\circ \theta^{-1}&=ehr^{str}.
\end{align*}\end{prop}

\begin{proof} Let $\iota$ be the character of $\h_\qp$ which sends any quasi-poset to $1$. 
By corollary \ref{cor20}, $\theta$ is an automorphism and $\theta^{-1}=\phi_{\iota^{*-1}}$. For any quasi-poset $P$:
\begin{align*}
\iota(1)&=1=ehr_P(1)=\sum_{\sim \triangleleft P} \frac{\mu_{P/\sim}}{cl(P/\sim)!}\alpha_{P|\sim}
=\sum_{\sim \triangleleft P} \lambda_{P/\sim}\alpha_{P|\sim}=\lambda*\alpha(P),
\end{align*}
so $\iota=\lambda*\alpha$; hence, its inverse is $\beta *\alpha^{str}$, and for any quasi-poset $P$,
as $cl(P/\sim)=cl(\sim)$ and $cc(P/\sim)=cc(P)$ for any $\sim\triangleleft P$:
\begin{align*}
\beta *\alpha^{str}(P)&=\sum_{\sim \triangleleft P}(-1)^{cl(\sim)+cc(P)}\frac{\mu_{P/\sim}}{cl(P/\sim)!} \alpha^{str}_{P|\sim}\\
&=(-1)^{cc(P)}ehr^{str}_P(-1)\\
&=(-1)^{cc(P)+cl(P)}ehr^{str}(1)\\
&=(-1)^{cc(P)+cl(P)}.
\end{align*}
Hence:
\begin{align*}
\theta^{-1}(P)&=Id \leftarrow (\beta *\alpha^{str})(P)
=\sum_{\sim \triangleleft P} (-1)^{cc(P|\sim)+cl(P|\sim)}P/\sim=\sum_{\sim \triangleleft P}  (-1)^{cl(\sim)+cl(P)} P/\sim.
\end{align*}
Moreover:
\begin{align*}
ehr^{str}\circ \theta&=(\phi_0\leftarrow \alpha^{str})\circ (Id\leftarrow \iota)\\
&=((\phi_0\leftarrow\alpha^{str})\circ Id)\leftarrow \iota\\
&=(\phi_0\leftarrow\alpha^{str})\leftarrow\iota\\
&=\phi_0\leftarrow (\alpha^{str}*\iota)\\
&=\phi_0\leftarrow (\alpha^{str}*\lambda*\alpha)\\
&=\phi_0\leftarrow \alpha\\
&=ehr.
\end{align*}\end{proof}

\subsection{A link with Bernoulli numbers}

For any $k\in \mathbb{N}$, let $c_k$ be the corolla quasi-poset  with $k$ leaves: $c_k=([k+1], \leq_{c_k})$,
with $1\leq_{c_k} 2,\ldots,k+1$:
\begin{align*}
c_0&=\tdun{$1$},&c_1&=\tddeux{$1$}{$2$},&c_2&=\tdtroisun{$1$}{$3$}{$2$},&c_3&=\tdquatreun{$1$}{$4$}{$3$}{$2$}\ldots
\end{align*}
B proposition \ref{prop34}, $\displaystyle \lambda_{c_k}=\frac{1}{k+1}$. Moreover:
\begin{align*}
L_{c_k}&=\{f:[k+1]\longrightarrow \mathbb{N}^*\mid f(1)\leq f(2),\ldots, f(k+1)\},\\
L_{c_k}^{str}&=\{f:[k+1]\longrightarrow \mathbb{N}^*\mid f(1)<f(2),\ldots, f(k+1)\},
\end{align*}
so, for all $n \geq 1$:
\begin{align*}
Ehr_{c_k}^{str}(n)&=(n-1)^k+\ldots+1^k=S_k(n),
\end{align*}
where $S_k(X)$ is the unique polynomial such that for all $n\geq 1$, $S_k(n)=1^k+\ldots+(n-1)^k$.
As a consequence, $\alpha_{c_k}^{str}$ is equal to the $k$-th Bernoulli number $b_k$. \\

Let $\sim\triangleleft c_k$. As the equivalence classes of $\sim$ are connected:
\begin{itemize}
\item The equivalence class of the minimal element $1$ of $c_k$ contains $i$ leaves, $0\leq i \leq k$.
\item The other equivalence classes are formed by a unique leaf.
\end{itemize}
Hence:
\begin{align*}
\delta(\isoclass{c_k})&=\sum_{i=0}^k \binom{k}{i} \isoclass{c'_{i,k-i}}\otimes \isoclass{c_i} \tun^{k-i},
\end{align*}
where $c'_{i,k-i}$ is the quasi poset on $[k+1]$ such that:
$$1\sim_{c'_{i,k-i}}\ldots \sim_{c'_{i,k-i}} i+1 \leq_{c'_{i,k-i}} i+2,\ldots k+1.$$
Hence, by theorem \ref{theo35}:
\begin{align*}
S_k(X)&=\sum_{i=0}^k \binom{k}{i} \lambda_{c'_{i,k-i}}b_iX^{k-i+1}\\
&=\sum_{i=0}^k \binom{k}{i} \lambda_{\overline{c'_{i,k-i}}}b_iX^{k-i+1}\\
&=\sum_{i=0}^k \binom{k}{i} \lambda_{c_{k-i}}b_iX^{k-i+1}\\
&=\sum_{i=0}^k \binom{k}{i} \frac{b_i}{k-i+1}X^{k-i+1}.
\end{align*}
We recover in this way Faulhaber's formula. For all $n\geq1$, $ehr_{c_k}(n)=n^k+\ldots+1^k$, 
and the duality principle gives, for all $n\geq 1$:
$$(-1)^{k+1}S_k(-n)=1^k+\ldots+n^k=S_k(n)+n^k.$$

\section{Noncommutative version}

\subsection{Reminders on packed words}

Let us recall the construction of the Hopf algebra of packed words $\WQSym$ \cite{NovelliThibon,NovelliThibon2}.

\begin{defi}\label{defi40}
Let $w=x_1\ldots x_n$ be a word which letters are positive integers.
\begin{enumerate}
\item We shall say that $w$ is a packed word if there exists an integer $k$ such that $\{x_1,\ldots,x_n\}=[k]$.
The set of packed words of length $n$ is denoted by $\bfPW(n)$; the set of all packed words is denoted by $\bfPW$.
\item There exists a unique increasing bijection $f:\{x_1,\ldots,x_n\}\longrightarrow [k]$ for a well-chosen $k$.
We denote by $Pack(w)$ the packed word $f(x_1)\ldots f(x_k)$. Note that $w$ is packed if, and only if, $w=Pack(w)$.
\item Let $I\subseteq \N$. Let $i_1<\ldots<i_p$ be the indices $i$ such that $x_i \in I$. We denote by 
$w_I$ the word $x_{i_1}\ldots x_{i_p}$.
\end{enumerate}\end{defi}

As a vector space, $\WQSym$ is generated by the set $\bfPW$. The product is given by:
$$\forall u\in \bfPW(k),\: \forall v\in \bfPW(l),\:
u.v=\sum_{\substack{w=x_1\ldots x_{k+l}\in \bfPW(k+l),\\ Pack(x_1\ldots x_k)=u,\\ Pack(x_{k+1}\ldots x_{k+l})=v}} w.$$
The unit is the empty word $1$. The coproduct is given by:
\begin{align*}
\forall w\in \bfPW, \:\Delta(w)&=\sum_{k=0}^{\max(w)} w_{\{1,\ldots,k\}}\otimes Pack(w_{\{k+1,\ldots,\max(w)\}}).
\end{align*}
For example:
\begin{align*}
(11).(11)&=(1111)+(1122)+(2211),\\
(11).(12)&=(1112)+(1123)+(2212)+(2213)+(3312),\\
(11).(21)&=(1121)+(1132)+(2231)+(3321),\\
(12).(11)&=(1211)+(1222)+(1233)+(1322)+(2311),\\
(12).(12)&=(1212)+(1213)+(1223)+(1234)+(1323)+(1324)\\
&+(1423)+(2312)+(2313)+(2314)+(2413)+(3412),\\
(12).(21)&=(1221)+(1231)+(1232)+(1243)+(1332)+(1342)\\
&+(1432)+(2321)+(2331)+(2341)+(2431)+(3421),\\ \\
\Delta(111)&=(111)\otimes 1+(111)\otimes 1,\\
\Delta(212)&=(212)\otimes 1+(1)\otimes (11)+1\otimes (212),\\
\Delta(312)&=(312)\otimes 1+(1)\otimes (21)+(12)\otimes (1)+1\otimes (312).
\end{align*}

\subsection{Hopf algebra morphisms in $\WQSym$}

\begin{prop}\label{prop41}
 The two following maps are surjective Hopf algebra morphisms:
\begin{align*}
EHR&:\left\{\begin{array}{rcl}
(\h_\QP,m,\Delta)&\longrightarrow&\WQSym\\
P&\longrightarrow&\displaystyle \sum_{w\in W_P}w,
\end{array}\right.\\
EHR^{str}&:\left\{\begin{array}{rcl}
(\h_\QP,m,\Delta)&\longrightarrow&\WQSym\\
P&\longrightarrow&\displaystyle \sum_{w\in W^{str}_P}w.
\end{array}\right.&
\end{align*} \end{prop}

\begin{proof} Let $P\in \QP(k)$, $Q\in \QP(l)$, and $w$ be a packed word of length $k+l$. Then:
\begin{itemize}
\item $w\in W_{PQ}$ if, and only if, $Pack(w_1\ldots w_k)\in W_P$ and $Pack(w_{k+1}\ldots w_{k+l})\in W_Q$.
\item $w\in W^{str}_{PQ}$ if, and only if, $Pack(w_1\ldots w_k)\in W^{str}_P$ and $Pack(w_{k+1}\ldots w_{k+l})\in W^{str}_Q$.
\end{itemize}
This implies that :
\begin{align*}
EHR(PQ)&=EHR(P)EHR(Q),&EHR^{str}(PQ)&=EHR^{str}(P)EHR^{str}(Q).
\end{align*}

Let $P\in \QP(n)$. We consider the two sets:
\begin{align*}
A&=\{(w,k)\mid w\in W_P, 0\leq k\leq \max(w)\},\\
B&=\{(O,w_1,w_2)\mid O\in Top(P), w_1\in W_{Pack(P_{\mid [n]\setminus O})}, w_2 \in W_{Pack(P_{\mid O})}\}.
\end{align*}
We define a bijection between $A$ and $B$ by $F(w,k)=(O,w_1,w_2)$,
where:
\begin{itemize}
\item $O=w^{-1}(\{k+1,\ldots,\max(w)\})$.
\item $w_1=Pack(w_{\{1,\ldots,k\}})$.
\item$w_2=Pack(w_{\{k+1,\ldots,\max(w)\}})$.
\end{itemize} 
Then:
\begin{align*}
\Delta \circ EHR(P)&=\sum_{(w,k)\in A} Pack(w_{\{1,\ldots,k\}}) \otimes Pack(w_{\{k+1,\ldots,\max(w)\}})\\
&=\sum_{(O,w_1,w_2)\in B}w_1\otimes w_2\\
&=(EHR \otimes EHR)\circ \Delta(P).
\end{align*}
So $EHR$ is a Hopf algebra morphism. In the same way, $EHR^{str}$ is a Hopf algebra morphism. \\

Let $w$ be a packed word of length $n$. We define a quasi-poset structure on $[n]$ by
$i\leq_P j$ if, and only if, $w_i\leq w_j$. Then $W^{str}_P=\{w\}$, so $EHR^{str}(P)=w$: $EHR^{str}$ is surjective.
If $w'\in W_P$, then $\max(w')\leq \max(w)$ with equality if, and only if, $w=w'$. Hence:
$$EHR(P)=w+\mbox{words $w'$ with $\max(w')<\max(w)$}.$$
By a triangular argument, $EHR$ is surjective. \end{proof}\\

\textbf{Examples}.
\begin{align*}
EHR(\tdun{$1$})&=(1),&EHR^{str}(\tdun{$1$})&=(1),\\
EHR(\tddeux{$1$}{$2$})&=(12)+(11),&EHR^{str}(\tddeux{$1$}{$2$})&=(12),\\
EHR(\tddeux{$2$}{$1$})&=(21)+(11),&EHR^{str}(\tddeux{$2$}{$1$})&=(21),\\
EHR(\tdun{$1$}\tdun{$2$})&=(12)+(21)+(11),&EHR^{str}(\tdun{$1$}\tdun{$2$})&=(12)+(21)+(11),\\
EHR(\tdun{$1,2$}\hspace{2mm})&=(11),&EHR^{str}(\tdun{$1,2$}\hspace{2mm})&=(11).
\end{align*}

\textbf{Remark.} The Hopf algebra $\WQSym$ has a polynomial representation \cite{NovelliThibon2}. Let $X=\{x_1,x_2,\ldots\}$
be an infinite, totally ordered alphabet; then, for any packed word $w$:
$$P_w(X)=\sum_{w'\in X^*,\: Pack(w')=w} w'.$$
With this polynomial representation, for any $P\in \QP(n)$:
\begin{align*}
EHR(P)&=\sum_{f\in L_P} x_{f(1)}\ldots x_{f(n)},&EHR^{str}(P)&=\sum_{f\in L^{str}_P} x_{f(1)}\ldots x_{f(n)}.
\end{align*}

\begin{prop}\label{prop42}
The following map is a Hopf algebra automorphism:
\begin{align*}
\Theta&:\left\{\begin{array}{rcl}
(\h_\QP,m,\Delta)&\longrightarrow&(\h_\QP,m,\Delta)\\
P&\longrightarrow&\displaystyle \sum_{\sim\triangleleft P} P/\sim.
\end{array}\right.
\end{align*}
Its inverse is:
\begin{align*}
\Theta^{-1}&:\left\{\begin{array}{rcl}
(\h_\QP,m,\Delta)&\longrightarrow&(\h_\QP,m,\Delta)\\
P&\longrightarrow&\displaystyle \sum_{\sim \triangleleft P}(-1)^{cl(P)+cl(\sim)}P/\sim.
\end{array}\right.
\end{align*}
Moreover, $EHR^{str}\circ \Theta=EHR$ and $EHR\circ \Theta^{-1}=EHR^{str}$.
\end{prop}

\begin{proof} By corollary \ref{cor20}, $\Theta=\phi_\iota$ is a Hopf algebra automorphism, where $\iota$ is defined in the proof
of proposition \ref{prop39}. Its inverse is $\phi_{\iota^{*-1}}$. \\

Let us prove that:
$$W_G=\sum_{\sim\triangleleft P} W_{G/\sim}^{str}.$$
Let $w \in W_G$; we define an equivalence $\sim_w$ by $x\sim_w y$ if $w(x)=w(y)$ and $x$ and $y$ are in the same connected component
of $w^{-1}(w(x))$. By definition, the equivalence classes of $\sim_w$ are connected.
If $x\sim_{P/\sim_w} y$, there exists $x_1,x'_1\ldots,x_k,x'_k,y_1,y'_1\ldots,y_l,y'_l$ such that:
\begin{align*}
&x\leq_P x_1\sim_w x'_1\leq_P\ldots \leq_P x_k\sim_w x'_k \leq_P y,\\
&y\leq_P y_1\sim_w y'_1\leq_P\ldots \leq_P y_l\sim_w y'_l \leq_P x.
\end{align*}
As $w\in W_P$, $w(x)\leq w(x_1)=w(x'_1)\leq \ldots \leq w(x_k)=w(x'_k)\leq w(y)$; by symmetry, $w(x)=w(x_1)=\ldots=w(x'_k)=w(y)=i$.
Moreover, as the equivalence classes of $\sim_w$ are connected, $x$ and $y$ are in the same connected component of $w^{-1}(i)$, so 
$x \sim_w y$: $\sim_w \triangleleft P$.

If  $x\leq_P y$ or $x\sim_w y$, then $w(x) \leq w(q)$. By transitive closure, if $x\leq_{P/\sim_w} y$, then $w(x)\leq w(y)$,
so $w\in W_{P/\sim_w}$. Moreover, if $w(i)\neq w(j)$, we do not have  $x\sim_w y$, so $w\in W^{str}_{P/\sim_w}$.

Let us assume that $\sim \triangleleft P$ and let $w\in W^{str}_{P/\sim}$. If $x\leq_P y$, then $x\leq_{P/\sim} y$, so $w(x)\leq w(y)$:
$W^{str}_{P/\sim}\subseteq W_P$.

Let us assume that $w\in W_{P/\sim}^{str}$, with $\sim \triangleleft P$. If $x \sim y$, then $w(x)=w(y)=i$
and $x$ and $y$ are in the same connected component of $P|\sim$, so are in the the same connected component of $w^{-1}(i)$:
$x\sim_w y$. If $x\sim_w y$, then $w(x)=w(y)=i$ and there exists $x_1,x'_1\ldots,x_k,x'_k$ with $w(x_1)=w(x'_1)=\ldots=w(x_k)=w(x'_k)=i$ such that:
$$x\leq_P x_1 \geq_P x'_1\leq_P\ldots \geq_P x'_k \leq_P y.$$
As $w \in W_{P/\sim}^{str}$, $x\sim_{P/\sim} x_1$, $x_1\sim_{P/\sim} x'_1,\ldots, x'_k \sim_{P/\sim} y$.
So $x\sim_{P/\sim} y$; as $\sim\triangleleft P$, $x\sim y$. Finally, $\sim=\sim_w$.  \\

We obtain that:
\begin{align*}
EHR(P)&=\sum_{w \in W_P}w=\sum_{\sim\triangleleft P}\sum_{w\in W_{P/\sim}^{str}} w
=\sum_{\sim \triangleleft P} EHR^{str}(P/\sim)=EHR^{str}(\Theta(P)).
\end{align*}
So $Ehr^{str}\circ \Theta=EHR$. \end{proof}

\begin{prop}
Let us consider the following map:
$$H:\left\{\begin{array}{rcl}
\WQSym&\longrightarrow&\K[X]\\
w\in \bfPW&\longrightarrow&H_{\max(w)}(X).
\end{array}\right.$$
This is a surjective Hopf algebra morphism, making the following diagram commuting:
$$\xymatrix{&\h_\QP\ar@{->>}[rd]^{EHR}\ar@{->>}[d]^{\Theta}\ar@{->>}[ld]_{\isoclass{}}&\\
\h_\qp\ar@{->>}[d]_{\theta}\ar@{->>}[rd]_(.3){ehr}&\h_\QP\ar@{->>}[ld]_(.3){\isoclass{}}|(.48)\hole \ar@{->>}[r]_{EHR^{str}}&\WQSym\ar@{->>}[ld]^{H}\\
\h_\qp\ar@{->>}[r]_{ehr^{str}}&\K[X]&}$$
\end{prop}

\begin{proof}
Let $P \in \QP$. Then:
\begin{align*}
ehr(\isoclass{P})&=\sharp W_P(k) H_k(X)=\sum_{w\in W_P} H_{\max(w)}(X)=\sum_{w\in W_P} H(w)=H\circ EHR(P).
\end{align*}
So $ehr\circ \isoclass{}=H\circ EHR$. Similarly, $ehr^{str}\circ \isoclass{}=H\circ EHR^{str}$.\\

Let us prove that $H$ is a Hopf algebra morphism. Let $w_1,w_2\in \WQSym$. There exist $x_1,x_2\in \h_\QP$,
such that $w_1=EHR(x_1)$ and $w_2=EHR(x_2)$. Then:
\begin{align*}
H(w_1w_2)&=H(EHR(x_1)EHR(x_2))\\
&=H\circ EHR(x_1x_2)\\
&=ehr(\isoclass{x_1x_2})\\
&=ehr(\isoclass{x_1})ehr(\isoclass{x_2})\\
&=H\circ EHR(x_1)H\circ EHR(x_2)\\
&=H(w_1)H(w_2).
\end{align*}
Let $w\in \WQSym$. There exists $x\in \h_\QP$ such that $w=EHR(x)$.
\begin{align*}
\Delta \circ H(w)&=\Delta \circ H\circ EHR(x)\\
&=(H\otimes H)\circ (EHR \otimes EHR)\circ \Delta(x)\\
&=(H\otimes H)\circ \Delta \circ EHR(x)\\
&=(H\otimes H)\circ \Delta(w).
\end{align*}
So $H$ is a Hopf algebra morphism. \end{proof}

\subsection{The non-commutative duality principle}

\begin{lemma}
The following map is an involution and a Hopf algebra automorphism:
\begin{align*}
\Phi_{-1}&:\left\{\begin{array}{rcl}
\WQSym&\longrightarrow&\WQSym\\
w&\longrightarrow&\displaystyle(-1)^{\max(w)} \sum_{\mbox{\scriptsize $\sigma:[\max(w)]\twoheadrightarrow [l]$, non-decreasing}} \sigma\circ w.
\end{array}\right.
\end{align*}
\end{lemma}

\begin{proof} Using the surjective morphisms $EHR^{str}$ and $ehr^{str}$, taking the quotients of the cointeracting bialgebras $(\h_\QP,m,\Delta)$
and $(\h_\qp,m,\delta)$, we obtain that $(\WQSym,m,\Delta)$ and $(\K[X],m,\delta)$ are cointeracting bialgebras, with the coaction defined by:
\begin{align*}
\rho&=(Id \otimes H)\circ \delta:\WQSym\longrightarrow \WQSym \otimes \K[X]
\end{align*}
For any packed word $w$:
$$\rho(w)=\sum_{\mbox{\scriptsize $\sigma:[k]\twoheadrightarrow [l]$, non-decreasing}}
\sigma \circ w\otimes H_{\max(Pack(w_{\mid (\sigma\circ w)^{-1}(1)}))}(X) \ldots H_{\max(Pack(w_{\mid (\sigma\circ w)^{-1}(l)}))}(X) .$$

Using proposition \ref{prop4}, for any $\lambda \in \K$, considering the character:
\begin{align*}
ev_\lambda&:\left\{\begin{array}{rcl}
\K[X]&\longrightarrow&\K\\
P&\longrightarrow&P(\lambda),
\end{array}\right. \end{align*}
we obtain an endomorphism $\Phi_\lambda$ of $(\WQSym,m,\Delta)$ defined by $\Phi_\lambda=Id \leftarrow ev_\lambda$.
 if $\lambda \neq 0$, $\Phi_\lambda$ is invertible, of inverse $\Phi_{\lambda^{-1}}$. For any packed word $w$, denoting by $k$ its maximum:
\begin{align*}
\Phi_\lambda(w)&=\sum_{\mbox{\scriptsize $\sigma:[k]\twoheadrightarrow [l]$, non-decreasing}}
H_{\max(Pack(w_{\mid (\sigma\circ w)^{-1}(1)}))}(\lambda) \ldots H_{\max(Pack(w_{\mid (\sigma\circ w)^{-1}(l)}))}(\lambda) \sigma \circ w.
\end{align*}
In particular, for $\lambda=-1$, for any $p\in \N$:
$$H_p(-1)=\frac{(-1)(-2)\ldots (-k)}{k!}=(-1)^k.$$
Hence:
\begin{align*}
\Phi_{-1}(w)&=\sum_{\mbox{\scriptsize $\sigma:[k]\twoheadrightarrow [l]$, non-decreasing}}
(-1)^{\max(Pack(w_{\mid (\sigma\circ w)^{-1}(1)}))+\ldots+\max(Pack(w_{\mid (\sigma\circ w)^{-1}(l)}))} \sigma \circ w\\
&=(-1)^k \sum_{\mbox{\scriptsize $\sigma:[k]\twoheadrightarrow [l]$, non-decreasing}} \sigma\circ w.
\end{align*}
Indeed, if $x\in (\sigma\circ w)^{-1}(p)$ and $y\in (\sigma \circ w)^{-1}(q)$, with $p<q$, then $\sigma \circ w(x)<\sigma \circ x(y)$;
as $\sigma$ is non-decreasing, $x<y$. So there exists $n_1<n_2<\ldots<n_l=k$ such that for all $p$,
the values taken by $w$ on $(\sigma \circ w)^{-1}(p)$ are $n_{p-1}+1,\ldots, n_p$.
Hence, the values taken by $Pack(w_{\mid (\sigma\circ w)^{-1}(p)})$ are $1,\ldots,n_p-n_{p-1}$, so:
$$\max(Pack(w_{\mid (\sigma\circ w)^{-1}(1)}))+\ldots+\max(Pack(w_{\mid (\sigma\circ w)^{-1}(l)}))
=n_1+n_2-n_1+\ldots+n_l-n_{l-1}=n_l=k.$$
In particular, $\Phi_{-1}$ is an involution and a Hopf algebra automorphism of $(\WQSym,m,\Delta)$. \end{proof}

\begin{theo}[Non commutative duality principle] \label{theo45}
For any quasi-poset $P\in \QP$:
\begin{align*}
EHR(P)&=(-1)^{cl(P)} \Phi_{-1}\circ ERH^{str}(P),&
EHR^{str}(P)&=(-1)^{cl(P)} \Phi_{-1}\circ ERH(P).
\end{align*}
\end{theo}

\begin{proof} We shall use the following involution and Hopf algebra automorphism:
\begin{align*}
\Psi&:\left\{\begin{array}{rcl}
\h_\QP&\longrightarrow&\h_\QP\\
p\in \QP&\longrightarrow&(-1)^{cl(P)}P.
\end{array}\right.
\end{align*}
Recall that the character $\iota$ of $\h_\QP$ sends any $P\in \QP$ to $1$. By the duality principle:
\begin{align*}
\iota \circ \Psi(P)&=(-1)^{cl(P)}=(-1)^{cl(P)}ehr(P)(1)=ehr^{str}(-1)=ev_{-1}\circ ehr^{str}(P).
\end{align*}
So $\iota \circ \Psi=ev_{-1}\circ ehr^{str}$.\\

Let $P\in \QP$. Recalling that if $\sim \triangleleft P$, $cl(P|\sim)=cl(P)$:
\begin{align*}
\delta \circ \Psi(P)&=(-1)^{cl(P)}\sum_{\sim\triangleleft P} P/\sim \otimes P|\sim=\sum_{\sim\triangleleft P} P/\sim \otimes (-1)^{cl(P|\sim)}P|\sim
=(Id \otimes \Psi)\circ \delta(P).
\end{align*}
So $\delta \circ \Psi=(Id \otimes \Psi)\circ \delta$. Hence, for any $x \in \h_\QP$:
\begin{align*}
EHR\circ \Psi(x)&=EHR^{str}\circ (Id \leftarrow \iota)\circ \Psi(x)\\
&=EHR^{str}(\Psi(x)_0) \iota \circ \Psi(x)_1\\
&=EHR^{str}(x_0)\iota \circ \Psi(x_1)\\
&=EHR^{str}(x_0) ev_{-1} \circ ehr^{str}(x_1)\\
&=EHR^{str}(x^{(1)}) ev_{-1}\circ EHR^{str}(x^{(2)})\\
&=EHR\leftarrow ev_{-1}(x)\\
&=(Id\leftarrow ev_{-1})\circ EHR^{str}(x)\\
&=\Phi_{-1}\circ EHR^{str}(x),
\end{align*}
where we denote $\delta(x)=x^{(1)}\otimes x^{(2)}$ and $\rho(x)=x_0\otimes x_1$. As $\Phi_{-1}$ and $\Psi$ are involutions,
$EHR^{str}\circ \Psi=\Phi_{-1}\circ EHR$. \end{proof}\\

In $E_{\K[X]\rightarrow \K[X]}$, putting $\phi_\lambda=Id\leftarrow ev_\lambda$, 
for any $P\in \K[X]$, $\phi_\lambda(P)=P(\lambda X)$. Moreover, as $H$ is compatible with the coactions:
\begin{align*}
H\circ \Phi_\lambda&=H\circ (Id \leftarrow ev_\lambda)=H\leftarrow ev_\lambda=(Id \leftarrow ev_\lambda)\circ H=\phi_\lambda \circ H,
\end{align*}
so:
\begin{align*}
ehr\circ \Psi&=H\circ EHR \circ \Psi=H\circ \Phi_{-1} \circ EHR^{str}=\phi_{-1} \circ H\circ EHR^{str}=\phi_{-1} \circ ehr^{str}.
\end{align*}
In other words, for any $P\in \QP$, $(-1)^{cl(P)}ehr_P(X)=ehr^{str}_P(-X)$: we recover the duality principle.\\

We obtain the commutative diagram of Hopf algebra morphisms:
$$\xymatrix{\h_\QP\ar@{^(->>}[d]_\Theta \ar@{->>}[rd]^(.6){EHR}\ar@{-->>}@/^1pc/[rrrrdd]_(.55){\isoclass{}}&&&&\\
\h_\QP\ar@{^(->>}[d]_\Psi\ar@{->>}[r]^(.4){EHR^{str}}\ar@{-->>}@/^1pc/[rrrrdd]_(.55){\isoclass{}}|(.32)\hole&\WQSym\ar@{^(->>}[d]_{\Phi_{-1}}
\ar@{-->>}[rrrrdd]^H|(.7)\hole&&&\\
\h_\QP\ar@{-->>}[r]^(.4){EHR}\ar@{-->>}@/^1pc/[rrrrdd]_(.55){\isoclass{}}&\WQSym\ar@{-->>}[rrrrdd]^H|(.7)\hole
&&&\h_\qp\ar@{^(->>}[d]_\theta \ar@{->>}[rd]^{ehr}&\\
&&&&\h_\qp\ar@{^(->>}[d]_\psi\ar@{->>}[r]_{ehr^{str}}&\K[X]\ar@{^(->>}[d]_{\phi_{-1}}\\
&&&&\h_\qp\ar@{->>}[r]_{ehr}&\K[X]}$$

\subsection{Compatibility with the other product and coproduct}

\begin{theo}
We define a second coproduct $\delta$ on $\WQSym$: 
$$\forall w\in \bfPW,\:\delta(w)=\sum_{(\sigma,\tau)\in A_w} (\sigma \circ w)\otimes (\tau \circ w),$$
where $A_w$ is the set of pairs of packed words $(\sigma,\tau)$ of length $\max(w)$ such that:
\begin{itemize}
\item $\sigma$ is non-decreasing.
\item If $1\leq i<j\leq \max(w)$ and $\sigma(i)=\sigma(j)$, then $\tau(i)<\tau(j)$.
\end{itemize}
Then $(\WQSym,m,\delta)$ is a bialgebra and $EHR^{str}$ is a bialgebra morphism from $(\h_\QP,m,\delta)$ to $(\WQSym,m,\delta)$.
\end{theo}

\begin{proof} Let us prove that $\delta \circ EHR^{str}=(EHR^{str}\otimes EHR^{str})\circ \delta$. Let $P\in \QP$. We consider the two following sets:
\begin{align*}
A&=\{(\sim,w_1,w_2)\mid \sim\triangleleft P, w_1\in W_{P/\sim}^{str}, w_2 \in W_{P|\sim}^{str}\},\\
B&=\{(w,\sigma,\tau)\mid w\in W_P^{str}, (\sigma,\tau)\in A_w\}.
\end{align*}

Let $(\sim,w_1,w_2)\in A$. We put $I_p=w_1^{-1}(p)$ for all $1\leq p\leq \max(w_1)$, and $w_2^{(p)}$ the standardization of the restriction
of $w_2$ to $I_p$. We define $w$ by:
$$w(i)=w_2^{(p)}(i)+\max w^{(2)}_1+\ldots+\max w^{(2)}_{p-1}\mbox{ if }i\in I_p.$$
Let us prove that $w\in W_P^{str}$. If $x\leq_P y$, then $x \leq_{P/\sim} y$, so $p=w_1(x)\leq w_2(y)=q$.
\begin{itemize}
\item If $p<q$, then $w(x)<w(y)$.
\item If $p=q$, then $w_1(x)=w_2(y)$ and, as $x\leq_P y$, $x$ and $y$ are in the same connected component of $w^{-1}(p)$.
So $x\sim_{w_1} y$, that is to say $x\sim y$ as $w_1 \in W_{P/\sim}^{str}$,  and $x\leq_{P|\sim} y$, which implies that $w_2(x)\leq w_2(y)$
and finally $w(x)\leq w(y)$.
\end{itemize}
Let us assume that moreover $w(x)=w(y)$. Then $p=q$ and necessarily, $w_2(x)=w_2(y)$. As $w_2 \in W_{P|\sim}^{str}$,
$x \sim_{P|\sim} y$, so $x\sim_P y$.

If $w(x)=w(y)$, then by definition of $w$, $w_1(x)=w_1(y)$. So there exists a unique $\sigma :[\max(w)]\longrightarrow [\max(w_1)]$,
such that $w_1=\sigma \circ w$. If $w(x)<w(y)$, then, by construction of $w$, $w_1(x)\leq w_1(y)$: $\sigma$ is non-decreasing.

There exists a unique $\tau:[\max(w)]\longrightarrow [\max(w_2)]$, such that $w_2=\tau \circ \sigma$. As $Pack(w_{\mid I_p})=Pack((w_2)_{\mid I_p})$
for all $p$, $\tau$ is increasing on $I_p$.

To any $(\sim,w_1,w_2)\in A$, we associate $(w,\sigma,\tau)=F(\sim,w_1,w_2)\in B$, such that $w_1=\sigma \circ \tau$, $w_2=\tau \circ \sigma$,
and $\sim=\sim_{\sigma \circ \tau}$. This defines a map $F:A\longrightarrow B$.\\

Let $(w,\sigma,\tau) \in B$. We put $G(w,\sigma,\tau)=(\sim,\sigma,\tau)=(\sim_{\sigma \circ w}, \sigma\circ w,\tau \circ w)$. 
If $x\leq_P y$, then $w(x)\leq w(y)$, so $w_1(x)=\sigma \circ w(x)\leq \sigma \circ w(y)=w_1(y)$. If moreover $w_1(x)=w_1(y)$,
then as $x\leq_P y$, $x$ and $y$ are in the same connected component of $w_1^{-1}(w_1(x))$, so $x\sim_{w_1} y$: $w_1 \in W_{P/\sim}^{str}$.

If $x\leq_{P|\sim} y$, then $x \sim_{w_1} y$ and $x\leq_P y$, so $w_1(x)=w_1(y)$ and $w_(x)\leq w(y)$. By hypothesis on $\tau$,
$\tau \circ w(x) \leq \tau\circ w(y)$, so $w_2(x)\leq w_2(y)$. If moreover $w_2(x)=w_2(y)$, by hypothesis on $\tau$,
$w(x)=w(y)$. As $w\in W_P^{str}$, $x\sim_P y$, so $x\sim_{P|\sim} y$: $w_2\in W_{P|\sim}^{str}$. 

We defined in this way a map $G:B\longrightarrow A$. If $(\sim,w_1,w_2)\in A$:
\begin{align*}
G\circ F(\sim,w_1,w_2)&=G(w,\sigma,\tau)=(\sim_{\sigma \circ w},\sigma \circ w,\tau \circ w)=(\sim_{w_1},w_1,w_2)=(\sim,w_1,w_2).
\end{align*}
So $G\circ F=Id_A$. If $(w,\sigma,\tau) \in B$:
\begin{align*}
F\circ G(w,\sigma,\tau)&=F(\sim_{\sigma \circ w},\sigma \circ w,\tau \circ w)=(w,\sigma,\tau).
\end{align*}
So $F\circ G=Id_B$: $F$ and $G$ are inverse bijections. \\

We obtain:
\begin{align*}
(EHR^{str}\otimes EHR^{str})\circ \delta(P)&=\sum_{(\sim,w_1,w_2)\in A} w_1\otimes w_2\\
&=\sum_{(w,\sigma,\tau)\in B} \sigma\circ w\otimes \tau \circ w\\
&=\sum_{w\in W_P^{str}} \delta(w)\\
&=\delta \circ EHR^{str}(P).
\end{align*}
So $EHR^{str}$ is compatible with $\delta$.\\

As $EHR^{str}$ is compatible with the product $m$ and the coproduct $\delta$, $Ker(EHR^{str})$ is a biideal of $(\h_\QP,m,\delta)$,
and $(\WQSym,m,\delta)$ is identified with the quotient $\h_\QP/Ker(EHR^{str})$, so is a bialgebra. \end{proof}\\

\textbf{Examples}.
\begin{align*}
\delta(11)&=(11)\otimes (11),\\
\delta(12)&=(12)\otimes ((11)+(12)+(21))+(11)\otimes (12),\\
\delta(21)&=(21)\otimes ((11)+(12)+(21))+(11)\otimes (21).
\end{align*}

This coproduct $\delta$ on $\WQSym$ is the internal coproduct of \cite{NovelliThibon2}, dual to the product of the Solomon-Tits algebra.\\

\textbf{Remarks.} \begin{enumerate}
\item The counit of $(\WQSym,m,\delta)$ is given by:
\begin{align*}
\varepsilon_B(w)&=\begin{cases}
1\mbox{ if }w=(1\ldots 1),\\
0\mbox{ otherwise}.
\end{cases} \end{align*}
\item There is no coproduct $\delta'$ on $\WQSym$ such that $(EHR\otimes EHR)\circ \delta=\delta'\circ EHR$. Indeed, if $\delta'$ is 
any coproduct on $\WQSym$, for $x=\tddeux{$1$}{$2$}+\tddeux{$2$}{$1$}-\tdun{$1$}\tdun{$2$}-\tdun{$1,2$}\hspace{2mm}$:
$$\delta'\circ EHR(x)=\delta'(0)=0,$$
but:
\begin{align*}
&(EHR\otimes EHR)\circ \delta(x)\\
&=(EHR\otimes EHR)((\tddeux{$1$}{$2$}+\tddeux{$2$}{$1$}-\tdun{$1$}\tdun{$2$})\otimes \tdun{$1$}\tdun{$2$}
+\tdun{$1,2$}\hspace{2mm}\otimes(\tddeux{$1$}{$2$}+\tddeux{$2$}{$1$}-\tdun{$1$}\tdun{$2$}-\tdun{$1,2$}\hspace{3mm}))\\
&=(11)\otimes (11).
\end{align*}\end{enumerate}

\begin{prop}
$H:(\WQSym,m,\delta)\longrightarrow (\K[X],m,\delta)$ is a bialgebra morphism. 
\end{prop}

\begin{proof} Let $w$ be a packed word. We denote $k=\max(w)$. Let $a,b\in \N$.
\begin{align*}
(H\otimes H)\circ \delta(w)(a,b)&=\sum_{(\sigma,\tau)\in A_w} H_{\max(\sigma \circ w)}(a)H_{\max(\tau \circ w)}(b)\\
&=\sum_{\mbox{\scriptsize $\sigma:[k]\twoheadrightarrow [l]$, non-decreasing}} \binom{a}{l}\binom{b}{|\sigma^{-1}(1)|}\ldots \binom{b}{|\sigma^{-1}(l)|}\\
&=\sum_{\substack{1\leq l\leq k,\\ i_1+\ldots+i_l=k}} \binom{a}{l}\binom{b}{i_1}\ldots \binom{b}{i_l}\\
&=\binom{ab}{k}\\
&=H_k(ab)\\
&=\delta(H(w))(a,b).
\end{align*}
As this is true for any $a,b\in \N$, $(H\otimes H)\circ \delta(w)=\delta \circ H$. \end{proof}

\begin{defi}
Let $w=w_1\ldots w_k$ and $w'=w'_1\ldots w'_l$ be two packed words. We put:
\begin{align*}
w\downarrow w'&=w_1\ldots w_k (w'_1+\max(w))\ldots (w'_l+\max(w)),\\
w \circledast w'&=w_1\ldots w_k (w'_1+\max(w)-1)\ldots (w'_l+\max(w)-1),\\
w\lightning w'&=w\downarrow w'+w\circledast w'.
\end{align*}
These three products are extended to $\WQSym$ by bilinearity.
\end{defi}

\begin{prop}
For all $x,y\in \h_\QP$:
\begin{align*}
EHR^{str}(x\downarrow y)&=EHR^{str}(x) \downarrow EHR^{str}(y),&
EHR(x\downarrow y)&=EHR(x) \lightning EHR(y).
\end{align*}\end{prop}

\begin{proof} Let $P\in \QP(k)$ and $Q\in \QP(l)$. If $w=w_1\ldots w_{k+l}$ is a packed word of length $k+l$:
\begin{align*}
w\in W^{str}_{P\downarrow Q}&\Longleftrightarrow w_1\ldots w_k\in L_P^{str}, w_{k+1}\ldots w_{k+l} \in L_Q^{str}, w_1,\ldots,w_k <w_{k+1},\ldots w_{k+l}\\
&\Longleftrightarrow w=w_P \downarrow w_Q, \mbox{ with }w_P\in W_P^{str}, w_Q \in W_P^{str}.
\end{align*}
So $W_{P\downarrow Q}^{str}=W_P^{str}\downarrow W_Q^{str}$, and:
\begin{align*}
EHR^{str}(P\downarrow Q)&=\sum_{w_P \in W^{str}_P,w_Q\in W^{str}_Q} w_P\downarrow w_Q=EHR^{str}(P)\downarrow EHR^{str}(Q).
\end{align*}
If $w=w_1\ldots w_{k+l}$ is a packed word of length $k+l$:
\begin{align*}
w\in W_{P\downarrow Q}&\Longleftrightarrow w_1\ldots w_k\in L_P, w_{k+1}\ldots w_{k+l} \in L_Q, w_1,\ldots,w_k \leq w_{k+1},\ldots w_{k+l}\\
&\Longleftrightarrow w=(w_P \downarrow w_Q, \mbox{ with }w_P\in W_P, w_Q \in W_P)\\
&\mbox{ or } w=(w_P \circledast w_Q, \mbox{ with }w_P\in W_P, w_Q \in W_P).
\end{align*}
Note that these two conditions are incompatible:
\begin{itemize}
\item in the first case, $\max(w_1\ldots w_k)=\min(w_{k+1}\ldots w_{k+l})-1$;
\item in the second case, $\max(w_1\ldots w_k)=\min(w_{k+1}\ldots w_{k+l})$.
\end{itemize}
So $W_{P\downarrow Q}=(W_P\downarrow W_Q)\sqcup (W_P\circledast W_Q)$, and:
\begin{align*}
EHR(P\downarrow Q)&=\sum_{w_P \in W_P,w_Q\in W_Q} w_P\downarrow w_Q+w_P\circledast w_Q\\
&=EHR(P)\downarrow EHR(Q)+EHR(P)\circledast EHR(Q),
\end{align*}
 so $EHR(P\downarrow Q)=EHR(P)\lightning EHR(Q)$. \end{proof}\\

\textbf{Remark.} As a consequence, $(\WQSym,\downarrow,\Delta)$ and $(\WQSym,\lightning,\Delta)$ are infinitesimal bialgebras
\cite{LodayRonco}, as $(\h_\QP,\downarrow,\Delta)$ is \cite{FMP,FM}.

\begin{cor}
For all $x,y\in \WQSym$:
\begin{align*}
\Phi_{-1}(x\downarrow y)&=\Phi_{-1}(x)\lightning \Phi_{-1}(y)&\Phi_{-1}(x\lightning y)&=\Phi_{-1}(x)\downarrow \Phi_{-1}(y).
\end{align*}\end{cor}

\begin{proof}
If $P,Q\in \QP$, then $cl(P\downarrow Q)=cl(P)+cl(Q)$, so:
$$\Psi(P\downarrow Q)=(-1)^{cl(P)+cl(Q)}P\downarrow Q=\Psi(P)\downarrow \Psi(Q).$$
Let $x,y\in \WQSym$. There exist $x',y'\in \h_\QP$, such that $EHR^{str}(x')=x$ and $EHR^{str}(y')=y$. Hence,
using the non-commutative duality principle:
\begin{align*}
\Phi_{-1}(x\downarrow y)&=\Phi_{-1}(EHR^{str}(x') \downarrow EHR^{str}(y'))\\
&=\Phi_{-1}\circ EHR^{str}(x'\downarrow y')\\
&=\Phi_{-1} \circ EHR^{str}\circ \Psi(\Psi(x') \downarrow \Psi(y'))\\
&=EHR(\Psi(x')\downarrow \Psi(y'))\\
&=EHR\circ \Psi(x') \lightning EHR\circ \Psi(y')\\
&=\Phi_{-1}(\Phi_{-1} \circ EHR \circ \Psi(x')) \lightning \Phi_{-1}(\Phi_{-1}\circ EHR \circ \Psi(y'))\\
&=\Phi_{-1}(EHR^{str}(x'))\lightning \Phi_{-1}(EHR^{str}(y'))\\
&=\Phi_{-1}(x)\lightning \Phi_{-1}(y).
\end{align*}
As $\Phi_{-1}$ is an involution, we obtain the second point. \end{proof}

\subsection{Restriction to posets}

In \cite{FM}, the image of the restriction  to $\h_\P$  of the map from $\h_\QP$ to $\WQSym$ defined by $T$-partitions
is a Hopf subalgebra, isomorphic to the Hopf algebra of permutations $\mathbf{FQSym}$ \cite{MR2,DHT}. This is not the case here:

\begin{prop}
$EHR(\h_\P)=EHR^{str}(\h_\P)=\WQSym$.
\end{prop}

\begin{proof} Let $w$ be a packed word of length $n$. We define a poset $P$ on $[n]$ by:
\begin{align*}
&\forall i,j\in [n],\: i\leq_P j\mbox{ if } (i=j)\mbox{ or} (w(i)<w(j)).
\end{align*} 
Note that if $i\leq_P j$, then $w(i)\leq w(j)$.
If $i\leq_P j$ and $j\leq_P k$, then:
\begin{itemize}
\item if $i=j$ or $j=k$, then obviously $i\leq_P k$.
\item Otherwise, $w(i)<w(j)$ and $w(j)<w(k)$, so $w(i)<w(k)$ and $i\leq_P k$.
\end{itemize}
Let us assume that $i\leq_P j$ and $j\leq_P i$. Then $w(i)\leq w(j)$ and $w(j)\leq w(i)$, so $w(i)=w(j)$. As $i\leq_P j$, $i=j$.
So $P$ is indeed a poset, and we observed that $w \in W_P$. \\

Let $w'$ be a packed word of length $n$. Let us prove that $w'\in W^{str}_P$ if, and only if, $w\leq w'$, where $\leq$ is the order on packed words
defined in definition \ref{defi28}. 

$\Longrightarrow$. Let us assume that $w'\in W^{str}_P$. If $w(i)<w(j)$, then $i\leq_P j$, so $w'(i)\leq w'(j)$.
Moreover, if $w'(i)=w'(j)$, then $i\leq_P j$, so $i=j$ as $P$ is a poset, and finally $w(i)=w(j)$: contradiction. So $w'(i)<w'(j)$,
we shows that $w\leq w'$.

$\Longleftarrow$. Let us assume that $w'\leq w$. If $i\leq_P j$, then $i=j$ or $w(i)<w(j)$, so $w'(i)=w'(j)$ ot $w'(i)<w'(j)$.
If, moreover, $w'(i)=w'(j)$, then $i=j$; so $w'\in W^{str}_P$.\\

We obtain an element $P\in \h_\P$ such that:
$$EHR^{str}(P)=\sum_{w\leq w'} w'.$$
As this holds for any $w$, by a triangularity argument, $EHR^{str}(\h_\P)=\WQSym$. By the non-commutative duality principle:
$$EHR(\h_\P)=\Phi_{-1}\circ EHR^{str}\circ \Psi(\h_\P)=
\Phi_{-1}\circ EHR^{str}(\h_\P)=\Phi_{-1}(\WQSym)=\WQSym,$$
as $\Phi_{-1}$ is an automorphism of $\WQSym$. \end{proof}

\bibliographystyle{amsplain}
\bibliography{biblio}

\end{document}